\documentclass[11pt]{amsart}
\usepackage{amsmath,amsthm,amscd,amssymb, color}
\usepackage{latexsym}
\usepackage[colorlinks, citecolor = blue,backref]{hyperref}
\usepackage[alphabetic]{amsrefs}
\usepackage[capitalize, nameinlink]{cleveref}
\usepackage{enumerate}
\usepackage[margin=1.25in]{geometry}
\usepackage{bbm}
\usepackage{mathtools}

\newcommand{\calH}{\mathcal H}
\newcommand{\Q}{\mathbb Q}
\newcommand{\C}{\mathbb C}
\newcommand{\Z}{\mathbb Z}

\renewcommand{\phi}{\varphi}
\newcommand{\eps}{\varepsilon}

\newcommand{\PP}{\mathbb P}
\newcommand{\frakH}{\mathfrak H}

\newcommand{\calO}{\mathcal O}
\newcommand{\calM}{\mathcal M}

\newcommand{\calA}{\mathcal A}
\newcommand{\bs}{\backslash}
\newcommand{\bmx}{\left( \begin{matrix}}
\newcommand{\emx}{\end{matrix} \right)}

\newcommand{\red}{\mathrm{red}} 
\renewcommand{\Im}{\mathrm{Im \, }}

\DeclareMathOperator{\SL}{SL}

\DeclareMathOperator{\End}{End}

\DeclareMathOperator{\Jac}{Jac} 
 
\DeclareMathOperator{\Gal}{Gal} 
\DeclareMathOperator{\Aut}{Aut} 
\DeclareMathOperator{\disc}{disc}

\newtheorem{lemma}{Lemma}
\numberwithin{lemma}{section}
\newtheorem{prop}[lemma]{Proposition}
\crefname{prop}{Proposition}{Propositions}
\newtheorem{thm}[lemma]{Theorem}

\theoremstyle{definition}
\newtheorem{ex}[lemma]{Example}

\theoremstyle{remark}
\newtheorem{rem}[lemma]{Remark}

\numberwithin{equation}{section}

\begin{document}

\title{Moduli for rational genus 2 curves with real multiplication for
discriminant 5}
\author{Alex Cowan}
\address{Department of Mathematics, Harvard University, Cambridge, MA 02138 USA}
\email{cowan@math.harvard.edu}

\author{Kimball Martin}
\address{Department of Mathematics, University of Oklahoma, Norman, OK 73019 USA}
\email{kimball.martin@ou.edu}

\date{\today}

\begin{abstract}
Principally polarized abelian surfaces with prescribed real multiplication (RM)
are parametrized by certain Hilbert modular surfaces.  
Thus rational genus 2 curves correspond to rational points
on the Hilbert modular surfaces via their Jacobians, 
but the converse is not true.  We
give a simple generic description of which rational moduli points 
correspond to rational curves, as well as give associated 
Weierstrass models, in the case of RM by the ring of integers
of $\Q(\sqrt 5)$. To prove this, we provide some techniques 
for reducing quadratic forms over polynomial rings.
\end{abstract}

\maketitle

%
%

\section{Introduction}

We are interested in describing the space of rational genus 2 curves
which have certain endomorphism structure on their Jacobians,
and will correspond to modular forms.

Let $k$ be a field.
Let $D > 0$ be a discriminant, and $\calO_D$ the quadratic order
of discriminant $D$.  For an abelian surface $A/k$,
if $\calO_D$ embeds in $\End_k(A)$, we say $A$ has real
multiplication (RM) by $\calO_D$, and abbreviate this as RM-$D$.
By extension, if $C$ is a genus 2 curve and $A = \Jac(C)$ has 
 RM-$D$,  we say $C$ has RM-$D$.

Typically Jacobians of genus 2 curves, and more generally
abelian surfaces, will have endomorphism ring $\Z$.  
One interest in abelian surfaces $A$ with RM 
(i.e., RM-$D$ for some $D$) is that they are
of GL(2) type, which by work of Ribet \cite{ribet} and
the proof of Serre's conjecture \cite{KW}, means that abelian
surfaces $A$ with RM over $k=\Q$ correspond to elliptic modular 
forms of weight 2.

Parametrizing genus 2 curves, with or without an RM condition,
is essentially understood over $k=\C$, but much less clear over
$k=\Q$.  In this paper, we give a relatively simple generic description of
moduli for genus 2 curves $C$ with RM-5 over $\Q$.

\begin{thm} \label{thm:main}
The $\C$-isomorphism classes of
genus $2$ curves $C/\Q$ with RM-5 are generically parametrized by
$(m,n) \in \Q^2$ such that $m^2-5n^2 - 5$ is a norm from 
$\Q(\sqrt 5)$.
\end{thm}

This parametrization is in terms of birational coordinates for
the Hilbert modular surface $Y(5)$, and the invariants for the
curve $C$ are then polynomial expressions in $m$ and $n$.
We will also describe models for these curves (\cref{prop:model}), 
and be more precise about the meaning of ``generically parametrized'' 
here (see \cref{thm:moduli} and \cref{sec:gh2mn}).
These results
extend to arbitrary subfields $k$ of $\C$, and the models
are rather simple when $k \supseteq \Q(\sqrt 5)$.

In order to explain our results more completely, we will first
describe moduli for genus 2 curves over $\C$ in more detail.
Below, when the field of definition of a curve or variety is not
specified, it is assumed to be $\C$.

Let $\calM_2$ be the (coarse) moduli space of genus 2 curves 
and $\calA_2$ be the moduli space of principally 
polarized abelian surfaces.  
The Torelli map $\calM_2 \to \calA_2$, corresponding
to sending a genus 2 curve $C$ to its Jacobian $A = \Jac(C)$,
is almost surjective---the complement of its image consists of
(moduli for) products of 2 elliptic curves.  
We may identify a point in $\calM_2$ corresponding to a genus
2 curve $C$ with Igusa--Clebsch invariants
$(I_2 : I_4 : I_6 : I_{10})$ in weighted projective space 
$\mathbb P^3_{1,2,3,5}(\C)$.  Each $(I_2 : I_4 : I_6 : I_{10})$
with $I_{10} \ne 0$ comes from a genus 2 curve.

The Igusa--Clebsch invariants $I_{2j}$ can be defined as degree $2j$
polynomial functions $I_{2j}(f)$ of the coefficients of a 
sextic Weierstrass equation $y^2 = f(x)$ for $C$, 
and up to projective equivalence do not depend on the model.   
Consequently, if $C$ has a model over a subfield $k \subseteq \C$, 
then the Igusa--Clebsch invariants are defined over $k$ (i.e., can
all be taken in $k$ after scaling).

However, the converse is not true.  
(Contrast this to the
genus 1 situation: an elliptic curve has a rational model if and only if
its $j$-invariant is rational.)
If $C$ is a genus 2 curve without extra automorphisms over 
$\C$ and its Igusa--Clebsch invariants
are defined over $k$, then Mestre \cite{mestre} showed that
$C$ is defined over $k$ if and only if a certain conic $L/k$ has a
$k$-rational point.  (If $C$ has extra automorphisms, it has a model
over $k$ by \cite{cardona-quer}.)
The coefficients of the Mestre conic $L$ are polynomials in $I_2, I_4, I_6$ and $I_{10}$.
Nonetheless, there is no simple characterization
of when the Mestre obstruction vanishes, i.e., when $L$ has 
a $k$-rational point.

Now we review moduli for genus 2 curves with RM-$D$ over $\C$.
For simplicity, assume $D$ is a fundamental discriminant, 
so $\calO_D$ is the ring of integers of $\Q(\sqrt D)$.  
The Hilbert modular surface
$Y_-(D)$ is a smooth compactification of the quotient
$\SL_2(\calO_D) \bs (\frakH^+ \times \frakH^-)$,
or alternatively
$\SL_2(\calO_D \oplus \calO_D^*) \bs (\frakH^+ \times \frakH^+)$,
where $\calO_D^*$ is the inverse different of $\calO_D$
(e.g., see \cite{vdG}).  Then $Y_-(D)$ is a coarse moduli
space for principally polarized abelian surfaces with real multiplication
RM-$D$, where one fixes an action of $\calO_D$ compatible
with the polarization.  

Suppose $k \subseteq \C$.
A genus 2 curve $C/k$ with RM-$D$ corresponds to a $k$-rational
point on $Y_-(D)$.  Again, the converse is not true.  If $p$
is a rational point on $Y_-(D)$ which does not correspond
to the product of 2 elliptic curves, then it corresponds to a curve
$C$ with RM-$D$ over $\C$.  For $p$ to correspond to
a curve over $k$ with RM-$D$ we need both that the Mestre
obstruction vanishes, and that some rational model for $C$ 
has RM-$D$ defined over $k$.  (It can happen that some 
$k$-rational models for $C$ have RM defined over $k$
and some do not.)  
We will see that generically if the Mestre obstruction vanishes, 
then the RM is defined over $k$.  More precisely, if
$\End(\Jac(C))$ is commutative, then a field of definition for $C$
is a field of definition for the RM (\cref{prop:RMdef}).
  
 \subsection{Strategy of proof}
 In the special case of RM-5, the Hilbert modular surface
$Y(5) = Y_-(5)$ is a rational surface, i.e., birational to 
$\mathbb P^2_{m,n}(\C)$.
 Hence to prove \cref{thm:main}, it suffices to show that
 the vanishing of the Mestre obstruction at a rational point $(m,n)$
 in $Y(5)$ is generically equivalent to the condition that
 $m^2 - 5n^2 - 5 = u^2 - 5v^2$ for some $u, v \in \Q$.
 This is not at all obvious from the Mestre conic,
 which is a conic over $\Q[m,n]$ whose coefficients are
 degree $\le 14$ polynomials in $m$ and $n$, and whose
 discriminant is of degree 30.  
 In fact, it was rather surprising to us that there was such a 
 simple characterization of the Mestre obstruction.  It was
 only through computational observations that we were led
 to believe in \cref{thm:main}, and then were able to find
 a proof after much trial.
 
 The starting point for the proof relies on two birational models for
 $Y_-(5)$ due to Elkies and Kumar \cite{EK}, which were obtained
 by studying lattice polarizations of K3 surfaces.  The
 first model is a double cover of $\mathbb P^2_{g,h}$ of the form
 $z^2 = f(g,h)$, where $f$ is a degree 5 polynomial in $g$ and $h$.
 In this model, the norm condition in \cref{thm:main} can be 
 restated as $30g+4$ being a norm from $\Q(\sqrt 5)$.  In particular,
 the Mestre obstruction only depends on $g$ and not $h$.
 (This was our initial computational observation that led to the
 theorem.)  The Igusa--Clebsch invariants now are low-degree
 expressions in $g$ and $h$.  In terms of $g$ and $h$,
 the Mestre conic has coefficients in $\Q[g,h]$ which are of degree 
 $\le 7$ in $g$ and degree $\le 2$ in $h$, and its discriminant
 is an integer multiple of $h^2(8h-9g^2)^2 z^2$.
 
 To our knowledge, there are no general methods to reduce
 quadratic forms over polynomial rings.
 The standard technique taught to ``simplify'' quadratic forms over
 fields is diagonalization, but unless one is very lucky this is
 not useful in simplifying quadratic forms over rings.  
 E.g., diagonalizing the conic over $\Q(m,n)$ and
 clearing denominators gives coefficients which are polynomials
 of degrees 24, 28 and 32 in $m$ and $n$.
 
 We describe a few simple techniques to reduce degrees of
 polynomial coefficients and 
 remove factors from the discriminant, which we hope may be of
 use in other situations.  In our case,
 we are able to use these methods to reduce the the Mestre conic
 in $g$ and $h$ to have polynomial coefficients of degree $\le 3$
 and remove the factors of $h^2$ and $(8h-9g^2)$ from the
 discriminant.  Then we switch to the $(m,n)$ model and
 apply our techniques to reduce the Mestre conic over
 $\Q(m,n)$ to $x_1^2 - 5x_2^2 + (m^2 - 5n^2 - 5)x_3^2 = 0$, which proves
 \cref{thm:main}.
 
 We remark that we needed to use both of these models
 for $Y_-(5)$ to carry out this reduction of the Mestre conic.
 While the Mestre conic is simpler in $g$ and $h$, our final
 reduced form, which is the same as $x_1^2 - 5x_3^2 + (30g+4)x_3^2 = 0$,
 is \emph{not} equivalent to the original Mestre conic over $\Q(g,h)$.
That is, these conics are not equivalent over $\Q$ for a generic choice
of $g, h \in \Q$---the equivalence requires rational $g, h$
such that $f(g,h)$ is a rational square, i.e., $g$ and $h$ come from a 
rational point on $Y_-(5)$, and it is not clear how to use the
relation $z^2 = f(g,h)$ to carry out this reduction solely in
terms of $g$ and $h$.  On the other hand, we were unable
to carry out the reduction entirely in terms of $m$ and $n$
because finding suitable changes of variables is more difficult
with higher degree polynomial coefficients.

\subsection{Moduli of rational curves}
Here we briefly describe to what extent we can make the
``generic'' aspect of \cref{thm:main} precise.  First, our reduction
of the Mestre conic $L$ over $\Q(m,n)$ does not give a 
$\Q$-equivalent conic when specializing to points $(m,n) \in \Q^2$
such that $\disc L = 0$.  This happens on a finite number of
curves in the moduli space, which we examine separately.
 
Second, as $(m,n)$ are only affine coordinates for a birational model 
for $Y_-(5)$, the set of rational $(m,n)$ does not exhaust
the rational points on $Y_-(5)$.  Fortunately, thanks to
work of Wilson \cite{wilson}, we can describe Igusa--Clebsch
invariants for the remaining points on $Y_-(5)$ and say
explicitly when such points correspond to a genus 2 curve defined
over $\Q$.

Consequently, in \cref{thm:moduli} we give an explicit description 
of a set $\mathcal Y$ of rational moduli in $\mathcal M_2$ such that 
any genus 2
 curve $C/\Q$ with RM-5 corresponds to a point on $\mathcal Y$.
 Moreover, any point in $\mathcal Y$ corresponds to a genus 2 curve
 $C/\Q$ that has potential RM-5, i.e., RM-5 defined over $\bar \Q$
 but not necessarily $\Q$.
 We do not know if each such $C$ will always have a twist with RM-5
 defined over $\Q$, but we were not able to find any examples to the
 contrary.
 At least the collection of such curves generically has RM-5, and we
 explain two ways in which one can check that the RM-5 is defined
 over $\Q$.

\subsection{Models of curves}
Several families of rational genus 2 curves $C/\Q$ with RM-5
have been constructed in the literature.  For instance,
Mestre constructed a 2-parameter family in \cite{mestre:family}
and Brumer constructed a 3-parameter family (see \cite{brumer}
for an announcement, and \cite{hashimoto}
for a proof different from Brumer's).  
For a rational choice of parameters these families
generically give rational genus 2 curves $C$ with RM-5 over $\Q$.
Moreover, over $\C$ these families are known to exhaust
all $\C$-isomorphism classes of genus 2 curves $C/\Q$ with
RM-5 (see \cite{hashimoto-sakai} for Brumer's family and
\cite{wilson} or \cite{sakai} for Mestre's family).  
However, it is not known how to describe all such rational
curves with these families, or
how to describe what parameters give $\C$-isomorphic curves.  

\cref{thm:main} generically parametrizes 
such $C/\Q$.  If $(m,n) \in \Q^2$ such that
$m^2-5n^2-5 = u^2-5v^2$ with $u, v \in \Q$, we give
a generic Weierstrass model $y^2 = f(x)$ for an associated
curve in terms of $(m,n,u,v)$.  See \cref{prop:model}.
These results apply arbitrary base fields $k \subseteq \C$.
If $k \supseteq \Q(\sqrt 5)$,
then the analogous norm condition in \cref{thm:main}
is automatically satisfied, and one can write down a 
model solely in terms of $(m,n) \in k^2$.  See
\cref{prop:modelQsqrt5}.

\subsection{Additional remarks}
Our original motivation for this project was to help understand
weight 2 elliptic modular forms with rationality field $\Q(\sqrt 5)$.
We hope to return to this in the future.

In \cref{sec:otherD}, we briefly describe some computational
evidence that there are similarly simple descriptions
for when the Mestre obstruction vanishes for some other
small values of $D$.  However, in these cases, the
Mestre conics that arise are more complicated and we have
only been partially successful in applying our reduction methods
to these cases.

Calculations for this project were carried out in
Sage \cite{sage} and Magma \cite{magma}.

\subsection*{Acknowledgements}
We are particularly grateful to Noam Elkies for many helpful
discussions and comments.  We also thank Armand Brumer and
John Voight for useful discussions.
Both authors were supported by grants from
the Simons Foundation (550031 for AC, and 512927 for KM).
Part of this work was carried out while the
second author was visiting MIT and Harvard, and he thanks
them for their hospitality.

%
%

\section{Moduli spaces}

Henceforth, $k$ denotes a subfield of $\C$.

Let $C$ be a genus 2 curve defined over $k$.  Then it
has a rational Weierstrass model of the form $y^2 = f(x)$, where
$f(x) \in k[x]$ is a sextic with no repeated irreducible factors.  
The Igusa--Clebsch
invariants $I_2, I_4, I_6, I_{10}$ are polynomial invariants of $f$ of
respective degrees $2, 4, 6, 10$ with $I_{10} = \disc(f)$.  We view
the Igusa--Clebsch invariants as a point $(I_2 : I_4: I_6: I_{10})$ in weighted
projective space $\PP^3_{1,2,3,5}$.  In this way, the Igusa--Clebsch invariants
in $\PP^3_{1,2,3,5}$ depend only on $C$ and not on the choice of the
Weierstrass equation.  Moreover, the set of $(I_2 : I_4: I_6: I_{10})$ with
$I_{10} \ne 0$ forms a coarse moduli space $\calM_2$ for genus 2 curves.

\subsection{Hilbert modular surfaces}
Here we review some facts about certain Hilbert modular surfaces.
See \cite{vdG} and \cite{EK} for more details.  

Let $D > 0$ be a fundamental discriminant.  The Hilbert 
modular surface $Y_-(D)$ is a smooth compactification 
of the quotient 
$\SL_2(\calO_D \oplus \calO_D^*) \bs \frakH^+ \times \frakH^+$.
When $K = \Q(\sqrt D)$ has narrow class number 1, this agrees
with the Hilbert modular surface often denoted $Y(D)$.

Fix an embedding $K \subseteq \C$ and
denote by $\tau$ the nontrivial Galois automorphism of $K$.
One can associate
to $(z_1, z_2) \in \frakH^+ \times \frakH^+$ a lattice
\[ L_{(z_1,z_2)} = \{ (a z_1 + b, a^\tau z_2 + b^\tau) :
a \in \calO_D, b \in \calO_D^* \} \subseteq V = \C^2. \]
Then
\[ E((w_1, w_2), (w'_1, w'_2)) = \frac{\Im w_1 \bar w_1}{\Im z_1}
+ \frac{ \Im w_2 \bar w'_2}{\Im z_2} \]
(with bar denoting complex conjugation)
defines a Riemann form on $A = V/L_{(z_1, z_2)}$ such that
$L_{(z_1,z_2)}$ is unimodular with respect to this form.  
This makes $A$ a principally polarized abelian surface (PPAS) 
with an action of $\calO_D$ via
$j(\alpha)(w_1, w_2) = (\alpha w_1, \alpha^\tau w_2)$.
In fact, one may check that $j : \calO_D \hookrightarrow \End(A)^\dagger$,
where $\dagger$ denotes the Rosati involution.
This construction leads to the fact that $Y_-(D)$ is a moduli
space for such pairs $(A,j)$ of PPASs with RM-$D$.

  The Humbert modular surface $\calH_D$ is the image of 
$Y_-(D)$ in $\calA_2$, and the map $Y_-(D) \to \calH_D$ is generically
2-to-1, corresponding to forgetting the action of $\calO_D$.
Note that in the above construction, switching $z_1$ and $z_2$
corresponds to replacing $j$ with $j \circ \tau$, and for the points
$(z_1, z_1)$, the conjugate actions $j$ and $j \circ \tau$ are isomorphic.

If $A$ is a geometrically simple PPAS, then $\End(A)$ is isomorphic to
$\Z$, an order in a real quadratic field, an order in a quartic CM field,
or an order in an indefinite quaternion algebra. 
If $A$ is not geometrically simple, but $\calO_D$ embeds in $\End(A)$,
then $\End(A)$ is an order in either the split quaternion algebra $M_2(\Q)$ or in $M_2(F)$ where $F$ is an imaginary quadratic field,
according to whether $A$ is isogenous over $\overline{\Q}$ to a product
of isogenous elliptic curves without or with CM.

\subsection{Fields of definition}

We are interested in fields of definition of curves and endomorphisms.
In general, suppose $X$ is a coarse moduli space space for a class of
varieties $V$ satisfying some property $P$.  If $x$ corresponds to the
pair $(V,P)$, then the field of moduli for $(V,P)$ is the field of definition
of the point $x$.  If both $V$ and $P$ are defined over $k$, then
the field of moduli contains $k$, but the converse is not true in general.

In particular, if $C$ is a genus 2 curve over $\C$, then the field of moduli
of $C$ is the field of definition of $(I_2 : I_4 : I_6 : I_{10})$, i.e. 
the minimal field $k_0$ such that 
$I_2, I_4, I_6, I_{10}$ can be taken in $k_0$ after scaling.
If $C$ is defined over $k$, then $k \supseteq k_0$.
However, $C$ need not be defined over $k_0$, i.e., there need not
be a curve $C'/k_0$ such that $C'$ and $C$ are isomorphic 
over $\C$.

Generically, $\Aut(C)$ is generated by the hyperelliptic involution on $C$. If $\lvert \Aut(C) \rvert > 2$, then by \cite{cardona-quer}, $C$
is defined over $k_0$.  When $\Aut(C) \simeq C_2$, Mestre \cite{mestre}
constructed a nonsingular conic $L/k_0$ such that $C$ is defined
over $k \supseteq k_0$ if and only if $L$ has a $k$-point. The coefficients 
of $L$ are polynomials in $I_2, I_4, I_6$ and $I_{10}$---see 
\cref{section:mestre-conic-background}
for details.   We remark that since $L$ always
has a point over a quadratic extension $k'/k_0$, $C$ is always definable
over a (in fact, infinitely many) quadratic extension(s) of $k_0$.

Now consider a genus 2 curve $(C,j)$ with RM-$D$, where 
$j$ is an embedding of $\calO_D$ into $\End(A)$, $A = \Jac(C)$, 
that respects the polarization as above.  Then the field of moduli for 
$(C,j)$ is the minimal field $k_0$ such that $(A,j)$ corresponds to
a $k_0$-rational point on $Y_-(D)$.  This necessarily contains
the field of moduli for $C$ regarded as a general genus 2 curve,
and thus may be regarded as the field of moduli for the RM.

If $(C,j)$ is defined over $k$, i.e., there is a model for $C$ 
defined over $k$ such that $j(\calO_D) \subseteq \End_k(A)$,
then $k \supseteq k_0$.  Conversely, given $k \supseteq k_0$,
we would like a way to determine whether $(C,j)$ is defined
over $k$.  Necessarily, $C$ must be defined over $k$, i.e.,
the Mestre conic $L$ must have a $k$-rational point.
The following says that, generically, when the Mestre conic
has a point the RM is also defined over $k$.

\begin{prop} \label{prop:RMdef} 
Suppose $p$ is a $k$-rational point on $Y_-(D)$
corresponding to a PPAS $A$ defined over $k$ 
with an embedding $j : \calO_D \hookrightarrow \End_\C(A)$.  
If $\End_\C(A)$ is commutative,
then $j(\calO_D) \subseteq \End_k(A)$.
\end{prop}

\begin{proof}
Let $\sigma \in G_k$, and $\eta = \frac{D + \sqrt D}2$.  
Then $p$ being $k$-rational means
there is an isomorphism $\phi : (A, j) \to (A^\sigma, j^\sigma)$.
In particular, $\phi$ maps $j(\eta)$ to $j^\sigma(\eta) \in \End_\C(A^\sigma)$, 
which we may identify with $j(\eta)^\sigma \in \End_\C(A)$.  
Consequently,
there is an inner automorphism of $\End_\C(A)$
taking $j(\eta)$ to $j(\eta)^\sigma$.  
Hence if $\End_\C(A)$ is commutative, this means 
$j(\eta)^\sigma = j(\eta)$
for all $\sigma$, and thus $j(\eta) \in \End_k(A)$.
\end{proof}

Now we briefly address how to check the field of definition
of RM for specific curves $C$.
Suppose $C$ is defined over $k$, and let $A = \Jac(C)$. 

Algorithms for numerically computing $\End_k(A)$ and $\End_\C(A)$
have been implemented in Magma, which one can use to
provably exhibit RM-$D$ using correspondences---e.g., see \cite{KM}
or \cite{CMSV}. 

In the case we consider in this paper, $D=5$, another criterion which
is simpler to provably verify was provided by Wilson:

\begin{prop}[\cite{wilson}] \label{prop:wilson}
Let $y^2 = f(x)$ be a sextic Weierstrass model over $k$ for a genus $2$
curve $C$ with potential RM-5, i.e., $C$ has RM-5 defined over $\C$.
Then $C$ has RM-5 (defined over $k$) if and only if 
$\Gal(f) = \Gal(f/k)$ is contained in a transitive copy of $A_5$ inside $S_6$.
\end{prop}

It is easy to verify whether $C$ has potential RM-5, because one
can check whether it comes from a point on $Y_-(5)$ via its
Igusa--Clebsch invariants.  In particular, if $C: y^2=f(x)$ is a genus 2
curve over $k$ with $\deg f = 6$, then $C$ has RM-5 (over $k$) if
and only if its Igusa--Clebsch invariants are of one of the types
listed below in \cref{prop:necess-cond} and $\Gal(f)$ lies in one of the
transitive copies of $A_5$ inside $S_6$.

\subsection{Moduli for RM-5} \label{section:moduli}

Elkies and Kumar \cite{EK} give the following birational model for $Y_-(5)$:
\begin{equation} \label{eq:EK}
 Y : z^2 = 2 (-972 g^{5} - 324 g^{4} - 27 g^{3} - 4500 g^{2} h - 1350 g h + 6250 h^{2} - 108 h).
\end{equation}
For $(z,g,h)$ on the surface $Y$ corresponding to a point on $\calM_2 \subseteq \calA_2$, the 
Igusa--Clebsch invariants are
\[ (I_2 : I_4: I_6: I_{10}) = \left(24 g + 6 : 9 g^{2} : 81 g^{3} + 18 g^{2} + 36 h : 4 h^{2}\right) \]

The surface $Y_-(5)$ is rational, and Elkies and Kumar give a birational map
between $Y$ and $\PP^2$, with affine coordinates $(m,n)$, via
\begin{align}\label{eq:gh2mn}
\nonumber 30g + 9 &= m^2 - 5n^2 \\
h &=  m \frac{(30g+9)(15g+2)}{6250} + \frac{9(250g^2 + 75g + 6)}{6250} \\
\nonumber z &= n \frac{(30g+9)(15g+2)}{25}.
\end{align}
These equations give invertible transformations between the affine coordinates
$(z,g,h)$ on $Y$ and $(m,n)$ on $\PP^2$ outside of the locus where
$g = \frac{m^2-5n^2-9}{30}$ is $- \frac 3{10}$ or $- \frac 2{15}$.

In an alternative approach, Wilson \cite{wilson} constructed 
a coarse moduli space for genus 2 curves $C$ with RM-5
with coordinates $(z_6 : s_2 : \sigma_5) \in \PP^2_{1,2,5}$ with $\sigma_5 \ne 0$ such
that
\[ (I_2 : I_4: I_6: I_{10}) = \left(-2s_2 + 2z_6^2 : \frac{(s_2 + 2z_6^2)^2}{16} 
: \frac{9z_6 \sigma_5 - 4I_4(3s_2 - 2z_6^2)}{16} : \frac{\sigma_5^2}{1024}\right). \]
Moreover if $C$ is defined over $k$, then so is $(z_6 : s_2 : \sigma_5)$
and the quantity
\[ \Delta' = 64 z_{6}^{6} s_{2}^{2} + 96 z_{6}^{4} s_{2}^{3} + 48 z_{6}^{2} s_{2}^{4} - 256 z_{6}^{5} \sigma_{5} + 8 s_{2}^{5} - 400 z_{6}^{3} s_{2} \sigma_{5} - 1000 z_{6} s_{2}^{2} \sigma_{5} + 3125 \sigma_{5}^{2} \]
must be a square in $k$.

One can translate Wilson's coordinates to the Elkies--Kumar coordinates via
\[ (g, h) = \left(- \frac{2z_6^2 + s_2}{12z_6^2} , \frac{\sigma_5}{64 z_6^5} \right). \]
We remark that under this change of coordinates, $\Delta' = 2^{10} z^2$, so the condition that $\Delta'$ is a square in $k$
is automatically satisfied when $(z,g,h)$ is a $k$-rational point on $Y$.

If $z_6 \ne 0$, we can assume $z_6 = 1$ and this relation gives a one-to-one correspondence
between $(g,h) \in \C^2$ and $(s_2, \sigma_5) \in \C^2$.  If
$z_6 = 0$, then the Igusa--Clebsch invariants of the point $(z_6 : s_2 : \sigma_5)$
must either be $(0 : 0 : 0 : 1)$ if $s_2 = 0$ or
\[ (I_2 : I_4: I_6: I_{10}) = \left(-8 : 1: -3: \frac{\sigma_5^2}{s_2^5}\right) \]
otherwise.  Hence any genus 2 curve with RM-5 either corresponds to a point
$(g,h) \in \C^2$ or has Igusa--Clebsch invariants of the form
 $(0 : 0 : 0 : 1)$ or $(8 : 1 : 3 : s)$ for $s \ne 0$.  
When $z_6 = 0$, $\Delta' = 8 s_{2}^{5} + 3125 \sigma_{5}^{2}$.
Thus $\Delta'$ being a square in $k$ means either $\sqrt 5 \in k$ if $s_2 =0$
or $3125s^2-8s$ is a square, where $s =  -\frac{\sigma_5^2}{s_2^5}$, 
if $s_2 \ne 0$.   It is easy to see that any two
 of these possibilities are mutually exclusive.
 
 Let us now consider the possibility that $(g,h)$ and $(g',h')$ give the
 same Igusa--Clebsch invariants, i.e., there exists $\lambda \in \C^\times$ such that
\[ \left(24 g' + 6 : 9 g'^{2} : 81 g'^{3} + 18 g'^{2} + 36 h' : 4 h'^{2}\right) = \lambda \cdot
 \left(24 g + 6 : 9 g^{2} : 81 g^{3} + 18 g^{2} + 36 h : 4 h^{2}\right) \]
 Since we are interested in genus 2 curves, assume $h$ and $h'$ are both nonzero.
 
 First note if $g=0$, then $g'=0$ and we have $h' = \lambda ^3 h$ and $h'^2 = \lambda^5 h^2$.  Comparing these shows $\lambda = 1$.  So assume $g, g'$ are
 both nonzero.  Then comparing $I_4$'s yields $\lambda = \eps \frac{g'}g$, where
 $\eps = \pm 1$.  Now comparing $I_2$'s shows $4g' + 1 = \eps(4g' +  \frac{g'}g)$.
If $\eps = 1$, then $g = g'$, i.e., $\lambda = 1$ which implies $h = h'$.
Thus assume $\eps = -1$.  Then $g' = - \frac{g}{8g+1}$ and $\lambda = \frac 1{8g+1}$.
Examining the $I_6$'s and $I_{10}$'s then gives
$h' = \frac{g^3+2h}{2(8g+1)^3}$ and 
\[ (h')^2 = \frac{(g^3+2h)^2}{(8g+1)^6} = \frac{4h^2}{(8g+1)^5}. \]
Using the assumption that $g \ne 0$, the latter equality holds if and only if
$32h^2 - 4g^2 h - g^5 = 0$, i.e., $h = \frac{g^2}{16}(1 + u)$ where $u^2 = 1+8g \ne 0$.
Note that if $g' =g$ then $g = - \frac 14$, $\lambda = -1$
so $h'^2 = -h^2$.

Hence for any $(g,h) = \left(g, \frac{g^2}{16}(1 \pm \sqrt{8g+1})\right)$ with 
$g \ne 0, -\frac 18$,
the pair $(g',h') = \left(-\frac g{8g+1},  \frac{g^3 +2h}{(8g+1)^3}\right)$ are
distinct coordinates with the same Igusa--Clebsch invariants,
and these are the only pairs of distinct $(g,h)$-coordinates with
this property.

Now suppose $(g,h)$ and $(g',h')$ are distinct $k$-rational pairs giving the same
Igusa--Clebsch invariants as above, with $u^2 = 8g+1$.   Expressing $g, g', h, h'$ in terms of $u$,
we see that, for both $(g,h)$ and $(g',h')$, the right hand side of \eqref{eq:EK}
is in the $k^\times$-square class of $-(43u^2 + 22u + 43)$.

The above discussion yields the following.

\begin{prop} \label{prop:necess-cond}
Let $C$ be a genus $2$ curve with RM-5 defined over $k$.  
Then the Igusa--Clebsch invariants of $C$ 
must be of one of the following types:

\begin{enumerate}
\item $(I_2 : I_4: I_6: I_{10}) = (0 : 0: 0: 1)$ when $\sqrt 5 \in k$;

\item $(I_2 : I_4: I_6: I_{10}) = (8 : 1 : 3 : s)$ for some nonzero $s \in k$ such that $3125s^2-8s$ is a square; or

\item $(I_2 : I_4: I_6: I_{10}) = \left(24 g + 6 : 9 g^{2} : 81 g^{3} + 18 g^{2} + 36 h : 4 h^{2}\right)$ for a $k$-rational solution $(z,g,h)$ to \eqref{eq:EK} with 
$h \ne 0$.
\end{enumerate}

The above three cases are mutually exclusive.  In case (2), $s$ is
unique.  In case (3), the pair $(g,h)$ is unique except in the case that 
$(g, h) = \left(\frac 18(u^2-1), \frac 1{1024}(u-1)^2(u+1)^3\right)$ for some 
$u \in k^\times \setminus \{ \pm 1 \}$ such that $-(43u^2 + 22u + 43)$
is a square, in which case
$(g, h)$ and $(g', h') = \left(- \frac{g}{8g+1}, \frac{g^3+2h}{2(8g+1)^3}\right)$ are distinct elements of $k^2$
that both correspond to invariants 
\[ \left(48u^2 + 16 : 36(1-u)^2(1+u)^2 : 
72(1-u)^2(1+u)^2(9u^2 + \frac 2 u + 9) : 4(1-u)^4(1+u)^6 \right). \]
\end{prop}

We remark that $-(43u^2 + 22u + 43)$ can be a square in
a number field $k$ if and only if every infinite place of $k$ is complex and the completion $k_v$ at every place $v$ above $3$ is an extension of $\Q_3$ of even degree.  In particular, when $k/\Q$ is quadratic this happens
if and only if $k$ is imaginary quadratic and non-split at 3.

\begin{rem} 
If we consider the map $\phi(u) = \left(\frac 18(u^2-1), \frac 1{1024}(u-1)^2(u+1)^3\right)$, then the pairs $(g,h)$ and $(g',h')$
yielding the same Igusa--Clebsch invariants at the end of the proposition
are just the points $\phi(u)$ and $\phi(\tfrac 1u)$, which both lie on the curve
$X_6 : 32h^2-4g^2h - g^5 = 0$ on $Y$.
Noam Elkies explained to us how his work in \cite{elkies} implies
that $X_6$ is the image of the Shimura curve quotient 
$X(6)/\langle w_6\rangle$ parametrizing
principally polarized abelian surfaces with quaternionic multiplication
by the maximal order in the rational quaternion algebra of discriminant 6.
Moreover, the involution on $X_6$ induced from $u \mapsto \tfrac 1u$
corresponds to the involution $w_2 = w_3$ of $X(6)/\langle w_6 \rangle$. 
\end{rem}

%
%

\section{Reduction of quadratic forms over polynomial rings}
\label{sec:reduction}

Here we will explain our approach to reducing quadratic forms over
polynomial rings, which we will then apply to Mestre conics.
Say $R = k[t_1, \dots, t_m]$ is a polynomial ring over a field $k$ 
of characteristic not 2.
Let $Q(x_1, \dots, x_n)$ be a quadratic form over $R$.  Thus
we can write $Q$ as
\[ Q(x_1, \dots, x_n) = \sum_{i, j} f_{i,j}(t_1, \dots, t_m) x_i x_j, \]
where each $A_{i,j} = f_{i,j}(t_1, \dots, t_m) \in R$ and
$A_{j, i} = A_{i, j}$.  Then $A = (A_{i,j}) \in M_n(R)$ is the Gram
matrix for $Q$ with respect to the standard basis $\{e_1, \dots, e_n \}$.
Define the polynomial degree $\deg_k Q$ of $Q$ to be 
$\max_{(i,j)} \deg A_{i,j}$.  

Consider the following two reduction problems:
(i) reduce $Q$ to an equivalent quadratic form $Q'$ over $R$
with minimal polynomial degree; or 
(ii) reduce $Q$ to a quadratic form $Q'$ over $R$ 
which is equivalent over the field of fractions $F$ of $R$
with minimal polynomial degree.
(By equivalence of quadratic forms, we mean isomorphism up to
invertible scaling.)
In case (i), specializations of $Q$ and $Q'$ to
any $t_1, \dots, t_m \in k$ will be $k$-equivalent.
In case (ii), specializations of $Q$ and $Q'$ will merely
be $k$-equivalent for generic choices of $t_1, \dots, t_m \in k$.

It is really reduction problem (ii) that we are interested in,
as it allows for much greater possibilities for reducing our
quadratic forms.  Note that merely diagonalizing $Q$ over
$F$ and clearing denominators to obtain a form over $R$
is not typically helpful in reducing the polynomial degree.
(Conversely, one cannot always diagonalize and maintain
minimal polynomial degree---see \cref{ex:red1}, but fortunately
for our Mestre conic of interest, our reduction process
will also diagonalize the form.) 
We first describe the types of reduction
steps we will use.

\begin{enumerate}
\item Simple degree reduction.  By a $k$-linear change of basis,
we may assume the maximal degree of the $f_{i,j}$'s is attained 
for some of the diagonal terms with $j=i$.  Say $f_{j_0,j_0}$ attains
the maximal degree of the $f_{i,j}$'s.
Write $v = \sum h_i(t_1, \dots, t_m) e_i$ where each $h_i \in R$.
Search for a choice of polynomials $h_i$ such that $\deg Q(v) <
\deg f_{j_0, j_0}$ and $h_{j_0}$ has nonzero constant term.
Now make the change of variable corresponding to changing
basis for the Gram matrix by replacing
$e_{j_0}$ in the standard basis with $v$.  The resulting quadratic 
form will have $Q(v)$ as the coefficient of $x_{j_0}^2$ and so we have
reduced the degree of this diagonal term.

In our Mestre conic case, the degrees of the diagonal terms
turn out to control the polynomial degree of $Q$, so reducing
degrees of diagonal terms is sufficient for us.  In general,
to reduce the degree of the $x_i x_j$ term, one could
similarly search for vectors $v$, $v'$ with polynomial
coefficients such that $\deg B(v, v') < \deg_k Q$,
and then change bases by replacing $e_i$ with $v$ and
$e_j$ with $v'$.

\item Discriminant reduction.  Let $\Delta = \Delta(t_1, \dots, t_m) \in R$
be the discriminant of $Q$.  By changing variables over $F$,
one may be able to remove polynomial factors from $\Delta$.
For instance, $Q_1 : x_1^2 + t_1 x_1 x_2 + t_1^2 x_2^2$ has
$\Delta = -3t_1^2$, and the change of variables 
$x_2 \mapsto \frac 1{t_1} x_2$ gives the quadratic form
$Q_2 : x_1^2 + x_1 x_2 + x_2^2$ with discriminant $-3$.
In general, since an invertible change of variables preserves the 
square class of the discriminant, we might hope to remove square
factors appearing in $\Delta$.  

First divide out any polynomial factors of the gcd of the coefficients
of $Q$.  
Now suppose $g(t_1, \dots, t_m) \in R$ is irreducible over $k$ of
positive degree such that $g^2 \mid \Delta$.  Then we can
attempt the following:

(a) Search for a polynomial vector $v$ such that $g^2 \mid Q(v)$,
with at least one of the coefficients of $v$ having a nonzero 
constant term (e.g., one can take $g(t_1) = t_1$ and $v = e_2$
with the above example of $Q_1$).
Then we can try a change of variables corresponding to
replacing some basis vector $e_i$ with $\frac{v}{g}$
where the $i$-th coefficient of $v$ has nonzero 
constant term.  This change of variables could introduce $g$
in the denominator of some $x_i x_j$ coefficients for $j \ne i$.
However, if we are fortunate, as always happens in our Mestre
conic reduction, then the resulting
quadratic form $Q'$ will still have coefficients in $R$,
and we will have removed a factor of $g^2$ from the discriminant.

(b) Assume $n \ge 3$, and if $n > 3$ that we have the
higher divisibility condition $g^{r} \mid \Delta$ for some
$r > \frac n2$.   Then one can
look for $F$-linearly independent vectors $v_1, \dots, v_{r} \in R^n$  
such that for each $1 \le i \le r$, $g \mid Q(v_i)$
but $g \nmid v_i$ (i.e., $g$ does not divide every 
polynomial coefficient of $v_i$).  Let $j_1, \dots, j_{n-r}$ 
be such that
that $e_{j_1}, \dots, e_{j_{n-r}}, v_1, \dots, v_{r} $ is a basis of $F^n$.  
Then the change of basis $\{ e_1, \dots, e_n \}$
to $\{ g e_{j_1}, \dots, g e_{j_{n-r}}, v_1, \dots, v_{r} \}$ transforms $Q$
to a quadratic form $Q'$ with an extra factor of $g^{2(n-r)}$ in its
discriminant, but now each coefficient of $Q'$ is
divisible by $g$.  Thus the $F$-equivalent form 
$g^{-1}Q'$ has coefficients in $R$, and we will have
removed a factor of $g^{2r-n}$ from $\Delta$.
\end{enumerate}

Simple degree reduction preserves $R$-equivalence,
whereas discriminant reduction only preserves 
$F$-equivalence.  Our strategy is to try simple
degree reduction, then discriminant reduction, and
repeat until the discriminant is squarefree, and then
finish with simple degree reduction.

First we give a baby example of simple degree reduction (1).
Below and in the next section, $e_1, \dots, e_n$ will denote 
the standard basis of the relevant vector space, and $A_i$ will denote
the Gram matrix for $Q_i$ with respect to $\{ e_1, \dots, e_n \}$.

\begin{ex} \label{ex:red1}
Let $R=\Q[t]$, and let $\{ e_1, e_2 \}$ be the standard basis for $M = R^2$.
Let $Q_1 = Q$ be the quadratic form on $M$ given by
\[ Q_1(x,y) = \left(t^{4} + 1\right) x^{2} + \left(2 t^{3} + 2 t\right) x y + \left(t^{2} - 1\right) y^{2}. \]
We can perform simple degree reduction as follows.  
We want to lower the degree of the
$x^2$-coefficient, so let $v = a_1 e_1 + (a_2 + b_2 t)e_2$.  
Then 
\[ Q_1(v) =  (a_1 + b_2)^2 t^4 + 2a_2(a_1+b_2) t^3  + 
(a_2^2 + 2a_1b_2 - b_2^2)t^2 + 2(a_1-b_2)a_2 t + (a_1^2 - a_2^2). \]
Hence setting $b_2 = -a_1$ makes $Q_1(v)$ a degree 2 polynomial
in $t$ with $t^2$-coefficient $(a_2^2-3a_1^2)$, which we cannot make
0 for nontrivial choices of $a_1, a_2 \in \Q$.  However, we can choose
to make either the $t^1$- or $t^0$-coefficient 0 by taking $a_2 = 0$ or
$a_2 = a_1$.  Let us take $v_1 = e_1 - te_2$ so $Q_1(v_1) = 
1-3t^2$, and let
$A_2$ be the Gram matrix for $Q_1$ with respect to $\{ v_1, e_2 \}$.
Let $Q_2$ be the associated quadratic form, i.e., the quadratic form
which has Gram matrix $A_2$ with respect to $\{ e_1, e_2 \}$.
In other words, $Q_2$ is obtained from $Q_1$ by the change of
variables $x \mapsto x$, $y \mapsto -tx + y$.
Then
\[ Q_2(x,y) = \left(1 -3 t^{2}\right) x^{2} + \left(4 t\right) x y + \left(t^{2} - 1\right) y^{2}. \]
Note that $Q_2$ has discriminant $12t^4 + 4$, so
we cannot hope to reduce the degree any further over $R$.

We remark that straightforward diagonalization of $Q_1$ gives
$ \left(t^{4} + 1\right) x^{2} + \frac{(1-3t^4)}{(t^4+1)}y^2$
and for $Q_2$ gives $\left(1 -3 t^{2}\right) x^{2} +\frac{3t^4+1}{3t^2-1} y^2$.
Since the discriminant is irreducible over $\Q$, one cannot
diagonalize over $R$ and have polynomial coefficients of degree $< 4$.
\end{ex}

A slightly more interesting example of (1) is given in the
reduction of the Mestre conic from $Q_1$ to $Q_2$ in 
\cref{subsection:conic-reduction-generic}.
Examples of (2a) are also given by the reductions from 
$Q_2$ to $Q_3$ and $Q_3$ to $Q_4$ in the
same section.   Then the reduction from $Q_5$ to $Q_6$
gives an example of (2b). 

All of these types of reduction involve finding polynomials
$h_i(t_1, \dots, t_m)e_i$ so that the coefficients
of $Q(v)$ satisfy certain conditions (e.g., no coefficients
above a certain degree, or whatever relations 
are imposed upon the coefficients by a divisibility
condition).
In general, this may be computationally challenging, 
as it involves 
finding simultaneous solutions of many quadratic 
equations in many variables to find suitable $h_i$'s.

As we do not have a general algorithm that will provably
minimize the polynomial degree, rather than trying to formulate
a precise reduction algorithm, we will just describe a few techniques
which can be used to lessen the computational difficulties
of these reduction steps in practice.  The first two techniques
apply to both (1) and (2).  The subsequent techniques
are just for discriminant reduction.

\begin{itemize}
\item \emph{Inductively try more complicated polynomial 
combinations of basis vectors.}  We begin by guessing 
certain forms for the polynomial coefficients $h_i$ of
$v$.  Each term of some $h_i$ with an unknown coefficient
adds another variable to solve for in finding a $Q(v)$
satisfying our desired criteria.  E.g., in \cref{ex:red1}
we need to make certain expressions in the
unknown coefficients $a_1, a_2, b_2$ zero to reduce the
degree.  To minimize the number of unknowns, we
begin by guessing as simple forms for the $h_i$'s as
we can hope for, and then try adding more terms as needed.

In \cref{ex:red1}, since we wanted to remove 
$t^4$ from the coefficient of $x^2$, and the coefficient of $y^2$
is degree 2 in $t$, it makes sense to consider constant
multiples $h_1(t)$ of $e_1$ plus linear multiples 
$h_2(t)$ of $e_2$ for $v$.  In fact, we might have first tried 
$h_1(t) = a_1$ and $h_2(t) = b_2 t$, and then if this
were not sufficient to remove the $t^4$ term, then 
we would try including a constant term in $h_2(t)$.  
If this were still unsuccessful, we could try letting $h_1(t)$
be a linear polynomial, which would necessitate $h_2(t)$
having degree 3.  While this is of course not needed such
simple examples as \cref{ex:red1}, it may be necessary
in the presence of additional variables (both more $x_i$'s  
and more $t_j$'s).

\item \emph{Look for coefficient conditions that factor.}
Say for instance that $m=2$, and we guess linear
forms $h_i(t_1,t_2) = a_i + b_i t_1 + c_i t_2$ for each $h_i$.
Then our desired conditions on $Q(v)$ may be something
like $\deg Q(v) < 4$ or $(t_1t_2 + 1)^2 \mid Q(v)$.
In the former case, say,  we want to make each 
$t_1^j t_2^{4-j}$ term of $Q(v)$ vanish.  That gives 5
quadratic equations in $3n$ unknowns.  How can we solve
this?

If our quadratic form is meant to reduce, we might hope
it does for algebraically simple reasons.  If we are fortunate,
then some of these quadratic equations we need to solve
may factor, as in the case of the $t^4$-coefficient of $Q_1(v)$
in \cref{ex:red1}.  If we are even more fortunate, this
forces one of our unknowns to be a certain linear combination
of other unknowns, and we can reduce the number of unknowns
and repeat.  We are fortunate in this way in the case of the
Mestre conic we reduce in 
\cref{subsection:conic-reduction-generic}.

\item \emph{Order of discriminant factor removal.}
In removing discriminant factors $g^r$, it may be easier to
remove certain factors before others.  On one hand, 
it may help to try to start with factors $g^2$ where $g$ is of small
degree, or $g$ only involves a small number of the variables 
$t_1, \dots, t_m$,
to more easily find $h_i$ such that $g^2 | Q(v)$ or $g | Q(v)$.
For instance, if $m=2$, $g(t_1, t_2) = t_1$ and we want
$g^2 | Q(v)$, then any $t_1^i t_2^j$ term in $Q(v)$ with 
$i \le 1$ must vanish.  However, the main issue we encountered
in reducing our Mestre conic was that, at a given stage,
attempting to remove one factor may lead to 
quadratic coefficient equations which factor, but attempting to
remove other factors does not. 

Thus for (2) we propose a process roughly of the following form.  
Try the simplest possible choices for $h_i$'s for removing 
different factors $g^r$ of the discriminant.  
Then pursue the ones that lead to linear relations among the
unknowns, inductively
adding more terms, and repeat until a factor is removed or
a bound for the complexity of the $h_i$'s is reached.
This approach is what led us the (otherwise unexplained) 
order of removing discriminant factors we use in 
\cref{subsection:conic-reduction-generic}.

\item \emph{Change variables to remove constant terms.}
If we want to remove a factor of say $(t_1-3)^2$ from
the discriminant, writing down the divisibility conditions
is a bit easier in practice if we first change
the polynomial variables $t_1 \mapsto t_1 + 3$, so
one is asking about removing a factor of $t_1^2$ from
a transformed form $Q'$.  For an example,
see the reduction of $Q_6$ in 
\cref{subsection:conic-reduction-generic}.

\item \emph{Examine minors.}
If some factor $g^r$ divides the discriminant of $Q$,
depending on $n$ and $r$, it may not be clear whether
we should try (2a) or (2b).  In this case, one can examine
the (determinant) minors of the Gram matrix.  
If some power of $g$ divides
sufficiently many minors, this suggest that (2b) may be possible.

Furthermore, if many of the diagonal minors are divisible
by $g$ then we can try looking for vectors $v_i$ as in (2b)
whose projection to $e_j$ is 0, for each $j$ in a set corresponding
to the minors.  E.g., if $r=n-1$ and each diagonal minor is divisible
is $g$, then we can look for vectors $v_1, \dots, v_{n-1}$ such that
the projection of $v_i$ to $e_i$ is 0 for each $i$.  This helps
reduce the number of unknowns we need to use, and is used
in the reduction of $Q'_6$ in 
\cref{subsection:conic-reduction-generic}.
\end{itemize}

%
%

\section{Reducing the Mestre conic} \label{section:conic-reduction}

\subsection{Mestre's construction of genus 2 curves} \label{section:mestre-conic-background}

Suppose $k \subseteq \C$, and $(I_2, I_4, I_6, I_{10}) \in k^4$ are 
Igusa--Clebsch invariants for a genus 2 curve $\mathcal C/\C$ without
extra automorphisms, i.e., $\Aut_\C(\mathcal C) \simeq C_2$.  
(In this section only, we use $\mathcal C$ rather than $C$ to denote
a genus 2 curve to avoid conflict with the notation for Clebsch
invariants.)
In \cite{mestre}, Mestre gave a method to determine whether $\mathcal C$
is defined over $k$, and if so, find a model.  Mestre worked in terms
of Clebsch invariants $(A,B,C,D)$ rather than Igusa--Clebsch invariants.
One can translate between these two sets of invariants via
\begin{align*}
&I_2 = -120  A, \quad I_4 = 90  (-8 A^{2} + 75 B), \quad
I_6 = 540  (16 A^{3} - 200 A B + 375 C) \\
&I_{10} = -162  (384 A^{5} - 6000 A^{3} B + 18750 A B^{2} - 10000 A^{2} C + 37500 B C + 28125 D).
\end{align*}

Mestre defines two elements $L$ and $M$ of $\Q(A,B,C,D)[x_1,x_2,x_3]$ as
$$L = \sum_{\substack{1 \leq i,j \leq 3}} L_{ij}x_ix_j\quad\text{and}\quad M = \sum_{\substack{1 \leq i,j,k \leq 3}} M_{ijk}x_ix_jx_k,$$
with
\begin{align*}
  &\quad L_{11} = 2C + \tfrac{1}{3}AB &\quad L_{22} &= D&\\
  &\quad L_{12} = \tfrac{2}{3}(B^2 + AC) &\quad L_{23} &= \tfrac{1}{3}B(B^2 + AC) + \tfrac{1}{3}C\left(2C + \tfrac{1}{3}AB\right)&\\
  &\quad L_{13} = D
  &\quad L_{33} &= \tfrac{1}{2}BD + \tfrac{2}{9}C(B^2 + AC),&
\end{align*}
\begin{flalign*}
  &\quad M_{111} = \tfrac{2}{9}\left(A^2C - 6BC + 9D\right)&\\
  &\quad M_{112} = \tfrac{1}{9}\left(2B^3 + 4ABC + 12C^2 + 3AD\right)&\\
  &\quad M_{113} = \tfrac{1}{9}\left(AB^3 + \tfrac{4}{3}A^2BC + 4B^2C + 6AC^2 + 3BD\right)&\\
  &\quad M_{122} = \tfrac{1}{9}\left(AB^3 + \tfrac{4}{3}A^2BC + 4B^2C + 6AC^2 + 3BD\right)&\\
  &\quad M_{123} = \tfrac{1}{18}\left(2B^4 + 4AB^2C + \tfrac{4}{3}A^2C^2 + 4BC^2 + 3ABD + 12CD\right)&\\
  &\quad M_{133} = \tfrac{1}{18}\left(AB^4 + \tfrac{4}{3}A^2B^2C + \tfrac{16}{3}B^3C + \tfrac{26}{3}ABC^2 + 8C^3 + 3B^2D + 2ACD\right)&\\
  &\quad M_{222} = \tfrac{1}{9}\left(3B^4 + 6AB^2C + \tfrac{8}{3}A^2C^2 + 2BC^2 - 3CD\right)&\\
  &\quad M_{223} = \tfrac{1}{18}\left(-\tfrac{2}{3}B^3C - \tfrac{4}{3}ABC^2 - 4C^3 + 9B^2D + 8ACD\right)&\\
  &\quad M_{233} = \tfrac{1}{18}\left(B^5 + 2AB^3C + \tfrac{8}{9}A^2BC^2 + \tfrac{2}{3}B^2C^2 - BCD + 9D^2\right)&\\
  &\quad M_{333} = \tfrac{1}{36}\left(-2B^4C - 4AB^2C^2 - \tfrac{16}{9}A^2C^3 - \tfrac{4}{3}BC^3 + 9B^3D + 12ABCD + 20C^2D\right),&
\end{flalign*}
and
\begin{align*}
  &L_{ij} = L_{ji}, \, M_{ijk} = M_{jik} = M_{ikj}.
\end{align*}
The \textit{Mestre conic} and the \textit{Mestre cubic} associated to
$\mathcal C$ (or equivalently, the Clebsch or Igusa--Clebsch invarants) 
are defined to be the projective varieties $L = 0$ and $M = 0$ over $\Q(A,B,C,D)$. In a slight abuse of terminology, we will occasionally say that $L$ itself is the Mestre conic, and similarly for $M$.

\begin{thm}[\cite{mestre}] \label{thm:mestre-conic}
  Suppose $(A,B,C,D) \in k^4$ are the Clebsch invariants of a
  genus 2 curve $\mathcal C/\C$ without extra automorphisms. 
  Then $\mathcal C$ is defined over $k$ if and only if the associated 
  the Mestre conic $L = 0$ in $\mathbb P^2(k)$ has
  a $k$-rational point.
\end{thm}
If the Mestre conic associated to $\mathcal C/\C$ has $k$-rational points then those rational points are parameterized by a single projective parameter which we will call $x$. We will write $x_i = x_i(x)$ with $i = 1,2,3$ to denote this parametrization.
\begin{thm}[\cite{mestre}] \label{thm:mestre-cubic}
 Suppose $(A,B,C,D) \in k^4$ are the Clebsch invariants of a
  genus 2 curve $\mathcal C/\C$ without extra automorphisms
  and the associated Mestre conic $L = 0$ has a $k$-rational point. 
Then a model for $\mathcal C$ over $k$ is given by
  $$y^2 = M(x_1(x), x_2(x), x_3(x)),$$
  where $M=0$ is the associated Mestre cubic.
\end{thm}

Finally, we elaborate on the condition that $\mathcal C/\C$ has
no extra automorphisms.  The possibilities for
extra automorphisms of genus 2 curves were determined by Bolza.
The reduced automorphism group of $\Aut^\red_\C(\mathcal C)$
is $\Aut_\C(\mathcal C)$ modulo the hyperelliptic involution.
If $\mathcal C$ has extra automorphisms, then 
$\Aut^\red_\C(\mathcal C)$ either contains an involution or has order 5.
The latter case happens exactly when
the (Clebsch or Igusa--Clebsch invariants) of $\mathcal C$
are $(0 : 0 : 0 : 1) \in \mathbb P^3_{1,2,3,5}$.

As explained in \cite{mestre},
the Mestre conic attached to a genus 2 curve $\mathcal C/\C$
is singular if and only if the reduced automorphism group of
$\mathcal C$ contains an involution.  Thus the condition that
$\mathcal C/\C$ has no extra automorphism can be restated
as: the Mestre conic $L=0$ is nonsingular and $I_2$, $I_4$ and $I_6$
are not all 0.

\subsection{The general case} \label{subsection:conic-reduction-generic}
Here we study the Mestre conic $L$ associated to a point $(z,g,h)$
of $Y$, i.e., to Igusa--Clebsch invariants 
$\left(24 g + 6 : 9 g^{2} : 81 g^{3} + 18 g^{2} + 36 h : 4 h^{2}\right)$.
After scaling by $2^4 \cdot 3^7 \cdot 5^{14}$, the Mestre conic $L : \sum_{i, j = 1}^3 L_{ij} x_i x_j = 0$ defined above has coefficients
\begin{align*}
L_{11} &= 189843750 (-96 g^{3} - 337 g^{2} - 108 g + 400 h - 9) \\
L_{12} &=  -2531250 (-144 g^{4} - 1299 g^{3} - 754 g^{2} + 2000 g h - 144 g + 500 h - 9) \\
L_{13} &= L_{22} =  -3750 (1944 g^{5} + 40905 g^{4} + 36990 g^{3} - 68400 g^{2} h + 11835 g^{2} - 43200 g h  \\ 
& \qquad \qquad  + 50000 h^{2}  + 1620 g - 5400 h + 81) \\
L_{23} &=  450 (324 g^{6} + 14931 g^{5} + 19395 g^{4} - 25800 g^{3} h + 9105 g^{3} - 30100 g^{2} h + 2020 g^{2} \\
& \qquad \qquad  - 8400 g h + 10000 h^{2} + 216 g - 700 h + 9) \\
L_{33} &= - (2916 g^{7} + 283338 g^{6} + 499041 g^{5} - 496800 g^{4} h + 319140 g^{4} - 915300 g^{3} h \\
& \qquad \qquad + 525000 g^{2} h^{2} + 101160 g^{3} - 426300 g^{2} h + 500000 g h^{2} + 17214 g^{2} - 76800 g h \\
& \qquad \qquad + 100000 h^{2} + 1512 g - 4800 h + 54) 
\end{align*}
The discriminant of $L$, by which we mean the determinant of the Gram
matrix, is then 
\[ \disc(L) = 2^7 \cdot 3^3 \cdot 5^{22} \cdot  h^{2} (8 h-9 g^{2})^{2} z^2. \]

Set $Q_1 = L$ and let $A_1$ be the Gram matrix of $Q_1$
with respect to the standard basis $\{e_1, e_2, e_3 \}$.
We will now perform a series of reductions on the Mestre conic
using the techniques described in the previous section.

Note that the $x_1^2$, $x_2^2$ and $x_3^2$ coefficients of
$L = Q_1$ are respectively degree 3, 5 and 7 polynomials in
$g$ (and degrees 1, 2 and 2 in $h$).
First we want to try to reduce the degree in $g$ of the $x_3^3$
coefficient.   Consider 
$v_1 = a_1 g^2 e_1 + a_2 g e_2 + e_3$, where $a_1$, $a_2$ denote
rational variables. The $Q_1(v_1)$ is degree 7 in $g$, and
the $g^7$-coefficient is $-2916(2500a_1 - 50a_2 + 1)^2$.  So set 
$a_2 = 50a_1 + \frac1{50}$.  This makes the $g^6$-coefficient of
$Q_1(v_1)$ equal
$-\frac{3^5 5^4}2(1250a_1 - 1)^2$.  Taking $a_1 = \frac 1{1250}$ gives
\begin{multline*} Q_1(v_1) = -2916 g^{5} - 24354 g^{4} + 10800 g^{3} h - 1500000 g^{2} h^{2} - 21483 g^{3} 
+ 78000 g^{2} h  \\
+  40000 g h^{2} - \frac{14259}{2} g^{2} + 39000 g h - 100000 h^{2} - 1026 g + 4800 h - 54, \end{multline*}
where
$v_1 = \frac 1{1250}g^2 e_1 + \frac 3{50}ge_2 + e_3$.
Thus we now consider the Gram matrix $A_2$ for $Q_1$
with respect to the basis $\{ e_1, e_2, v_1 \}$.  Let $Q_2$ be 
the resulting quadratic form from this change of variables, i.e., 
$Q_2(v) = {}^t v A_2 v$.  In particular, the $x_3^2$-coefficient
of $Q_2$ is $Q_1(v_1)$.

The $x_1^2$, $x_2^2$ and $x_3^2$ coefficients of $Q_2$ are 
degrees 3, 5 and 5 in $g$ (and no other coefficient has higher
degree).  We may try to reduce the coefficient degrees for
$x_2^2$ and $x_3^2$ by replacing $e_2$ and $e_3$
with vectors of the form $a_1 g e_1 + e_2$ and
$b_1 g e_1 + e_3$.  In this way, one to reduce $Q_2$
to a quadratic form whose coefficients are elements of
$\Q[g,h]$ of degree $\le 4$, but there are no obvious
ways to further reduce the degree from there, and this
reduction does not make the next step
any easier, so we will not do this.

Instead, we will next remove a polynomial factor from
the discriminant of $Q_2$, which is  a rational multiple 
of $h^2(8h-9g^2)^2 z^2$.  The $h^2$ factor has the
lowest degree, so we will begin with that.  We will find a vector
$v_2 = (a_1 + b_1 g)e_1 + a_2 e_2 + a_3 e_3$ such that
$Q_2(v_2)$ is divisible by $h^2$.  This is essentially the simplest
polynomial combination of standard basis vectors where we
can hope to kill off all of the $g^j$ terms in $Q_2(v_2)$, and it
turns out to be sufficient. 

The constant term of $Q_2(v_2)$ is 
$-54 \, {\left(5625 \, a_{1} - 75 \, a_{2} + a_{3}\right)}^{2}$, so we
set $a_3 = 75 a_2 - 5625a_1$.  Now we kill off the highest 
degree $g^j$ terms.
Then the $g^5$-coefficient of
$Q_2(v_2)$ is $-1822500  {\left(225  a_{1} - a_{2} - 100 b_{1}\right)}^{2}$.  Set $a_2 = 225a_1 - 100b_1$.  Then the $g^4$-coefficient 
is $-118652343750 {\left(3 a_{1} - b_{1}\right)}^{2}$.  Setting
$b_1 = 3a_1$ yields $Q_2(v_2)$ is a multiple of $h^2$.  Specifically,
take $a_1 = 2^{-2} \cdot 3^{-2} \cdot 5^{-6}$, and
then $Q_2(v_2) = -2(300 g^{2} + 2 g + 3)h^{2}$,
where $v_2 = \frac 1{562500}((1+3g)e_1 - 75e_2 - 11250e_3)$.
Let $A_3$ be the Gram matrix of $Q_2$ with respect to the basis
$\{ e_1, e_2, \frac 1h v_2 \}$, and $Q_3$ the associated quadratic
form.  So $Q_3$ is not $\Q[g,h]$-equivalent to $Q_2$,
but after specializing to any $g, h \in \Q$ with $h \ne 0$, the forms
$Q_2$ and $Q_3$ are $\Q$-equivalent.

Now we will remove the $(8h-9g^2)^2$ factor from the determinant.
The degrees in $g$ of the $x_1^2$, $x_2^2$ and $x_3^2$ coefficients
of $Q_3$ are 3, 5 and 3.
Let $v_3 = (a_1 + b_1g)e_1 + a_2 e_2 + (a_3 + b_3g)e_3$.
We want $Q_3(v_3)$ to be a multiple of $(8h-9g^2)^2$.
To kill the constant term of $Q_3(v_3)$, we need to set
$a_3 = 225a_2 - 16875a_1$.  Then to kill the $h$-coefficient,
we need $a_2 = 75a_1$.  Then to kill the $g^2$-coefficient,
$b_3 = 67500a_1 - 16875b_1$.  At this point there are only 
nonzero $g^5$, $g^4$ and $g_2h$ and $h^2$ terms, so for
$Q_3(v_3)$ to be a multiple of $(8h-9g^2)^2$ it needs to be
a rational multiple and the $g^5$ term must vanish.  This is
accomplished with $b_1 = \frac 32 a_1$.
In summary 
$v_3 = a_1 ((1+\frac 32 g)e_1 + 75e_2 + \frac{84375}2g e_3)$.
Taking $a_1 = 2 \cdot 3^{-1} \cdot 5^{-6}$ then gives
$Q_3(v_3) = -30(8h-9g^2)^2$.  Now let $A_4$ be the 
Gram matrix of $-Q_3$ with respect to the basis
$\{ \frac {e_1}{1875}, \frac {v_3}{8h-9g^2}, e_3 \}$,
and $Q_4$ the associated quadratic form,
\begin{multline*}
Q_4: \left(5184 g^{3} + 18198 g^{2} + 5832 g - 21600 h + 486\right) x_{1}^{2} + \left(612 g + 108\right) x_{1} x_{2} + 30 x_{2}^{2} \\+ \left(288 g^{2} + 684 g - 4000 h + 108\right) x_{1} x_{3} + \left(-240 g + 12\right) x_{2} x_{3} + \left(600 g^{2} + 4 g + 6\right) x_{3}^{2}.
\end{multline*}
Specializing $g, h$ to any rationals such that $h \ne 0$
and $8h \ne 9g^2$, $Q_4$ is $\Q$-equivalent to the original
Mestre conic $L$.  The discriminant of $Q_4$ is $-9600z^2$. 

Now there is no obvious way to further reduce the degree, 
and indeed, it seems that there is not much further simplification
that can be done over $\Q(g,h)$.  The reduction we
perform next will not preserve $\Q$-equivalence of quadratic
forms (even assuming $z \ne 0$) if $g, h$ are rational 
but $z$ is not.

Let $Q_5$ be the quadratic form over $\Q[m,n]$ obtained by
converting $Q_4$ from $(g,h)$ to $(m,n)$ via \eqref{eq:gh2mn}.
Let $A_5$ be the Gram matrix of $Q_5$ with respect
to the standard basis.  
The coefficients of $Q_5$ are elements of
$\Q[m,n]$ of degree $\le 6$, and the discriminant is 
\[ -\frac{96}{25} n^{2}  (m^{2} - 5 n^{2})^{2} (m^{2} - 5 n^{2} - 5)^{2}. \]  
Let $v_5 = a_1 e_1 + a_2 e_2 + a_3 e_3$
be a rational linear combination of the standard basis vectors.
Then $Q_5(v_5)$ has constant term
$\frac{3}{250}  {\left(63  a_{1} - 50  a_{2} - 70 a_{3}\right)}^{2}$.
Setting $a_3 = \frac{1}{70}(63  a_{1} - 50  a_{2})$, then gives that
$Q_5(v_5) = p_1(m,n) (m^2-5n^2)$ for some polynomial
$p_1(m,n)$ with constant term 
$5  {\left(441  a_{1} - 50  a_{2}\right)}^{2}$.
Hence we set $a_1 = 50$ and $a_2 = 441$ (which makes $a_3 = -270$) 
to get
\[ Q_5(v_5) = 30  {\left(16  m^{2} - 80  n^{2} + 2729\right)} {\left(m^{2} - 5 n^{2}\right)}^{2}. \]
Let $A_6$ be the Gram matrix of $Q_5$ with respect to the basis
$\{ \frac{v_5}{m^2-5n^2}, e_2, e_3 \}$, and $Q_6$ the associated
quadratic form, which has polynomial degree 4 and discriminant
$-9600 n^{2}  (m^{2} - 5 n^{2} - 5)^{2}$.

Next one might try to remove the $n^2$ factor from the discriminant,
but evaluating $Q_6$ on simple combinations such as 
$v = (a_1 + b_1 m) e_1 + a_2 e_2 + a_3 e_3$ leads to polynomials
without linear factors in $a_1, a_2, a_3, b_1$ for the coefficients
of powers of $m$.  So it is not immediately clear how to find some
$v$ such that $Q_6(v)$ is divisible by $n^2$.  On the other hand,
the diagonal minors of $A_6$ are divisible by $(m^{2} - 5 n^{2} - 5)$
which suggests we can remove a factor of $(m^{2} - 5 n^{2} - 5)$
from the discriminant by working with combinations of just 2 of the
standard basis vectors at a time.

To make it easier to look for multiples of $(m^{2} - 5 n^{2} - 5)$,
we first make the change of variables $m=m+5$, $n = n+2$.
This changes $(m^{2} - 5 n^{2} - 5)$ to $(m^{2} - 5 n^{2} + 10 m - 20 n)$,
which has no constant term.  Let $Q'_6$ be the resulting
quadratic form.  Now one can look for rational linear combinations $v$
of pairs of the basis vectors $e_1, e_2$ and $e_3$ such that
$Q_6'(v)$ has no constant term.  In particular, 
$u_1 = e_1 - 53 e_2$ and $u_2 = 11 e_2 - 15 e_3$ work and
both $Q_6'(u_1)$ and $Q_6'(u_2)$ are divisible by 
$(m^{2} - 5 n^{2} + 10 m - 20 n)$.  Let $A_7'$ be the
Gram matrix for $\tfrac{1}{6}(m^{2} - 5 n^{2} + 10 m - 20 n)^{-1}Q_6'$ with 
respect to the basis 
$\{ \frac {u_1}4, \frac{u_2}5, (m^{2} - 5 n^{2} + 10 m - 20 n)e_2 \}$,
and $Q_7'$ the resulting quadratic form.  Let $Q_7$ and $A_7$
denote the result of reverting $Q_7'$ and $A_7'$ back to our original
variables $m=m-5$, $n = n-2$.  Then $Q_7$ is:
\begin{multline*} Q_7 : 5x_{1}^{2} + 2m x_{1} x_{2} + \left(m^{2} - 5 n^{2} - 4\right) x_{2}^{2}  \\
+ \left(4 m^{2} - 20 n^{2} - 20\right) x_{2} x_{3} + \left(5 m^{2} - 25 n^{2} - 25\right) x_{3}^{2} .
\end{multline*}
The discriminant of $Q_7$ is 
$-25  n^{2}  (m^{2} - 5 n^{2} - 5)$.

Now we can use a vector of the form
$v = (a_1 + b_1 m) e_1 + a_2 e_2 + a_3 e_3$ to removed the
$n^2$ factor from the determinant.  Explicitly, by zeroing out
of coefficients of powers of $m$ in $Q_7(v)$, we find that
$Q_7(v_7) = -25 n^2$ where $v_7 = m e_1 - 5e_2 + 2e_3$.
Let $A_8$ be the Gram matrix for $\tfrac{1}{5} Q_7$
with respect to the basis $\{ e_1, \frac{v_7}{n}, e_3 \}$.
Then the associated quadratic form is
\[ Q_8 :  x_1^2 - 5 x_2^2 + (m^2 - 5n^2 - 5) x_3^2. \]
For any $m, n \in \Q$ such that $\disc L \ne 0$
the form $Q_8 \in \Q[x_1, x_2, x_3]$ is similar to $Q_1$.
Thus for such $m, n$, the Mestre conic $L$ has a rational point
if and only if $\pm (m^2 - 5n^2 - 5)$ is a norm from $\Q(\sqrt 5)$.
(Note that $-1$ is a norm from $\Q(\sqrt 5)$.)
Consequently, the analogue holds for any extension $k \supseteq \Q$.

\subsection{Points at infinity} \label{sec:infty}
Here we consider $k$-rational points on $Y_-(5)$ not coming from
affine coordinates $(z,g,h) \in Y$.  By \cref{prop:necess-cond},
if such a point corresponds to a genus 2 curve $C$,
there are two possibilities for the Igusa--Clebsch invariants:
(1) $(0 : 0 : 0 : 1)$ when $\sqrt 5 \in k$, and
(2) $(8 : 1 : 3 : s)$ where $s \in k^\times$ and $3125s^2-8s$
is a square in $k$.

We wish to determine when $C$ can be defined over $k$
in these cases.  In case (1), $C$ is already defined over $\Q$
with a model $y^2 = x^5 - 1$.  So we only need to analyze
case (2).

Let us consider the Mestre conic for Igusa--Clebsch invariants
as in (2).
After replacing $x_1$ with $2^{-1} \cdot 3^2 \cdot 5^3 x_1$, 
$x_2$ with $2 \cdot 3^3 \cdot 5^5 x_2$ and $x_3$ with 
$2^2 \cdot 3^4 \cdot 5^7 x_3$, the Gram matrix $A_1$ for the 
Mestre conic $Q_1 = L$ is
\[ A_1 = \left(\begin{array}{rrr}
-1 & 2 & -3125 s + 2 \\
2 & -6250 s + 4 & 4 \\
-3125 s + 2 & 4 & -43750 s + 4
\end{array}\right) \]
Then
\[ \det A = 2 \cdot 5^{10} (3125s - 8) s^2. \]

Then
$Q_1(a_1e_1 + a_2e_2 + a_3e_3)$ has constant term 
$-(a_1 - 2 a_2 + 2  a_3)^{2}$.
Now letting $A_2$ be the Gram matrix with respect to the basis
$\{ e_1, \tfrac{1}{25}(2e_1 + e_2), \tfrac{1}{125}(2e_1-e_3) \}$, we see
\[ A_2 = \left(\begin{array}{rrr}
1 & 0 & 25 s \\
0 & -10 s & 2 s \\
25 s & 2 s & -2 s
\end{array}\right). \]

Scale $A_2$ by $s$ and replace $x_2$ and $x_3$ with $x_2/s$ and $x_3/s$,
respectively, to get the equivalent Gram matrix
\[ A_3 = \left(\begin{array}{rrr}
-s & 0 & 25 s \\
0 & -10 & 2 \\
25 s & 2 & -2
\end{array}\right) \]
with quadratic form
\[ Q_3:  -s x_{1}^{2} -10 x_{2}^{2} + 50 x_{1} x_{3} + 4 x_{2} x_{3} -2 x_{3}^{2} \]
This has determinant $2 s (3125s - 8)$, and the
associated quadratic form is $\Q$-equivalent to the diagonal form
\[ 2 x_1^2 + 5s x_2^2 - (3125s - 8) x_3^2. \]
Assuming that $s(3125s-8)$ is a square in $k^\times$, this form is
$k$-equivalent to the forms
\[ 2 x_1^2 + 5s x_2^2 - s x_3^2 \sim 2s x_1^2 - (x_3^2 - 5x_2^2). \]
Clearly, this has a rational point if and only if  $2s$ is a norm from $k(\sqrt 5)$
(which is automatic if $\sqrt 5 \in k$).

%
%

\section{Moduli for rational curves}

Here we state our main result and complete the proof.

\begin{thm} \label{thm:moduli} 
Let $C$ be a genus $2$ curve with RM-5 defined over $k$.
Then the Igusa--Clebsch invariants $(I_2 : I_4 : I_6 : I_{10}) \in 
\mathbb P^3_{1,2,3,5}$ are of one of the following forms:

\begin{enumerate}
\item $\left(24 g + 6 : 9 g^{2} : 81 g^{3} + 18 g^{2} + 36 h : 4 h^{2}\right)$ for a $k$-rational solution $(z,g,h)$ to \eqref{eq:EK} such
that such that $30g+4$ is a norm from $k(\sqrt 5)$ and $hz(8h-9g^2) \ne 0$;

\item $(8 : 1 : 3 : s)$ where $s \in k^\times$
such that $s(3125s-8)$ is a square in $k^\times$ and
 $2s$ is a norm from $k(\sqrt 5)$;

\item $(12  (4 g + 1) : 36  g^{2} : 36  (18 g + 13)  g^{2} : 162  g^{4})$ where $g \in k^\times$ such that
$-3(128g+9)$ is a square in $k$;

\item \begin{multline*}
 \left( 20 (2 m^{2} - 3) : 25  (m - 3)^{2}  (m + 3)^{2} : \right. \\
\left.
5  (m + 3)^{2}  (75 m^{4} - 378 m^{3} + 428 m^{2} + 474 m - 711) : 8  (m - 2)^{4}  (m + 3)^{6}) \right)
\end{multline*}
where $m \in k$ or $m = \sqrt 5$;

\item $(8 : 1 : 3 : \frac 8{3125})$;

\item $(0 : 0 : 0 : 1)$ if $\sqrt 5 \in k$.\end{enumerate}
Cases (1) and (2) correspond to $\Aut_\C^\red(C) = \{ 1 \}$.
Cases (3)--(5) correspond to $\Aut_\C^\red(C)$ containing an involution,
and case (6) corresponds to $\# \Aut_\C^\red(C) = 5$.

Conversely, if $C$ is a genus $2$ curve over $\C$ with Igusa--Clebsch
invariants in one of the forms (1)--(6), 
then $C$ can be defined over $k$ and $C$ has potential
RM-5.  Moreover, in case (1), if $\End_\C(\Jac(C))$ is commutative,
then $C$ has RM-5 defined over $k$.
\end{thm}

In this theorem, by ``a norm from $k(\sqrt 5)$'' we mean in the image of
the relative norm map from $k(\sqrt 5)$ to $k$.  Thus such a norm condition
is automatically satisfied if $\sqrt 5 \in k$.

In \cref{sec:gh2mn}, we reformulate condition (1) in terms of $(m,n)$,
which removes the need to check \eqref{eq:EK} to determine the existence
of a $k$-rational point $(z,g,h) \in Y$ given $g, h \in k$.

\begin{rem} Suppose $k=\Q$ now, and that $C$ is a genus 2 curve over
$\Q$ with Igusa--Clebsch invariants of one of the forms
(1)--(5) in the theorem.  We would like to be able to say
when $C$ (or a twist) actually has RM-5 defined over $k$.  
Write $A = \Jac(C)$.
Generically, $\End_\C(A) \simeq \Z[\frac{1+\sqrt 5}2]$ in case (1) so
the RM-5 will be defined over $\Q$, but there 
 seems to be no simple way to describe the moduli points where
 $A$ has (split or non-split) quaternionic multiplication over $\C$.  
 We do not know whether the RM-5 must be defined over $\Q$
 if $\End_\C(A)$ is not commutative.
 
 In case (2), we also expect that generically 
 $\End_\C(A) \simeq \Z[\frac{1+\sqrt 5}2]$, and the RM will
 be defined over $\Q$---however we have not checked that
 the points satisfying (2) always correspond rational points on
 $Y_-(5)$ so cannot apply \cref{prop:RMdef}.  Still, one can
 check in examples for case (2), e.g., $s=\frac 2{25}$, that
 one gets a genus 2 curve with RM-5 defined over $\Q$.
 
In cases (3)--(5), $\End_\C(A)$ is an order in a $2 \times 2$ matrix algebra,
hence not commutative, and \cref{prop:RMdef} does not apply.  
Here $C$ has more than just quadratic twists, and one twist
may have RM-5 defined over $\Q$, and another may not.
For instance if $C : y^2 = x^6 - x^4 + 4x^2 -1$, which has
 Igusa--Clebsch invariants $(88 : 169 : 28561 : 57122)$
 corresponding to $g=-\frac{13}{96}$ in case (3),
 then $C$ has RM-5 defined over $\Q$, but the twist
 corresponding to $x \mapsto \sqrt{-1} x$ does not.
 This may be checked, for instance, by computing Galois
 groups and using \cref{prop:wilson}.
We do not know whether there will always be some
twist with RM-5 defined over $\Q$ in these cases.
\end{rem}

\begin{proof}
Suppose $C$ is a genus 2 curve over $\C$ with RM-5
and Igusa--Clebsch invariants $(I_2 : I_4 : I_6 : I_{10})$
defined over $k$.  If $C$ as well as the RM-5 is defined over $k$, then by
\cref{prop:necess-cond}, we know that the Igusa--Clebsch invariants
must be of the form $(0 : 0 : 0 : 1)$ (only if $\sqrt 5 \in k$), 
$(8 : 1 : 3 : s)$ for $s \in k^\times$
such that $3125s^2-8s$ is a square in $k$, or they correspond to a
$k$-rational point $(z,g,h) \in Y$ with $h \ne 0$, so we may assume
our Igusa--Clebsch invariants take one of these forms.

As explained earlier, 
$C$ has a model over $k$ if and only if the Mestre conic has a 
$k$-rational point or $C$ has extra automorphisms.  
Thus, to prove both directions of the theorem, it will suffice to show 
that: (i) when $C$ has no extra automorphisms, 
the Mestre conic has a $k$-rational point exactly in cases (1) and (2);
and (ii) the Igusa--Clebsch
invariants from \cref{prop:necess-cond} corresponding to
curves with extra automorphisms are described exactly by cases 
(3)--(6).

If the Igusa--Clebsch invariants are $(0 : 0 : 0 : 1)$, then $C$ has a model 
over $\Q$ given by $y^2 = x^5 - 1$, and the RM-5 is defined over 
$\Q(\sqrt 5)$.  This verifies the theorem (in both directions) 
for these Igusa--Clebsch invariants, i.e., case (6).

Assume now that the Igusa--Clebsch invariants come from a $k$-rational point
$(z,g,h) \in Y$ with $h \ne 0$. 
 
First suppose $z(8h-9g^2) \ne 0$, so that the Mestre conic is nonsingular
and $C$ has no extra automorphisms.  
Then the reduction we performed
over $\Q$ in \cref{subsection:conic-reduction-generic} implies that 
the Mestre conic has a $k$-rational point if and only if $30g+4 = m^2 -5n^2 -5$
is a norm from $k(\sqrt 5)$, except in the two special cases 
$g \in \{ - \frac 3{10}, - \frac 2{15} \}$, where there is not a one-to-one
correspondence between the $(z,g,h)$ and $(m,n)$ coordinates.
In \cref{sec:gh2mn} below, we check that one has a $k$-rational $(z,g,h) \in Y$
for $g \in \{ - \frac 3{10}, - \frac 2{15} \}$ if and only if $\sqrt 5 \in k$,
and in this case the Mestre conic always has a $k$-rational point.
This, together with \cref{prop:RMdef}, proves (both directions of)
the theorem in case (1).

For the cases where the Mestre conic is singular, 
the $k$-rational $(z,g,h) \in Y$ with $z(8h-9g^2) = 0$
correspond to Igusa--Clebsch invariants of the forms in cases (3) and (4).
The details are given in \cref{sec:mest-sing}.

Finally, suppose the Igusa--Clebsch invariants are of the form $(8:1:3:s)$,
where $s \in k^\times$ and $3125s^2-8s$ is a square in $k$.
Both directions of the theorem in case (2) follows from the reduction
of the Mestre conic in \cref{sec:infty}.  Case (5) follows from 
\cref{sec:mest-sing}.
\end{proof}

\subsection{Translation to  $(m,n)$-coordinates} \label{sec:gh2mn}
Here we explain how to translate \cref{thm:moduli} into the rational
model $\mathbb P_{m,n}^2$ for $Y_-(5)$, and treat the
exceptional cases $g \in \{ - \frac 3{10}, - \frac 2{15} \}$ in the proof of
\cref{thm:moduli}.

Recall that there is a one-to-one correspondence between $k$-rational 
coordinates $(z,g,h) \in Y$
and $k$-rational coordinates $(m,n) \in \mathbb A^2$
such that
$g = \frac{m^2-5n^2-9}{30} \not \in \{ - \frac 3{10}, -\frac 2{15} \}$.

If $g = -\frac 3{10}$, the equation for $Y$ becomes
$z^2 = \frac{4}{3125}(3125h - 27)^2.$
Hence there are no $k$-rational points $(z,-\frac 3{10},h)$ on $Y$
if $\sqrt 5 \not \in k$.  If $\sqrt 5 \in k$, then for all $h \in k$,
there is a $k$-rational $(z,-\frac 3{10},h) \in Y$.
Here the associated Mestre conic is nonsingular if $h \not \in \{ 0,
\frac{81}{800}, \frac{27}{3125} \}$, and always has a $k$-rational point.
For instance, in Sage we find the $k$-rational point
$(\frac{64000}{81}h^2 - \frac{2368}{75}h + \frac{284}{3125} :\frac{128}{15}h - \frac{51}{1250}: h + \frac 9{1000})$.

 If $g =  -\frac 2{15}$, the equation for $Y$ becomes
$z^2 = \frac{4}{3125}(3125h - 2)^2.$
Similarly, there are no $k$-rational points $(z,-\frac 2{15},h)$ on $Y$
if $\sqrt 5 \not \in k$, but if $\sqrt 5 \in k$, then there is a $k$-rational
 $(z,-\frac 2{15},h) \in Y$ for all $h \in k$.
 The associated Mestre conic is nonsingular if
 $h \not \in \{ 0, \frac{1}{50}, \frac{2}{3125} \}$
Again, one may check in Sage that the Mestre conic always
has a rational point.

These calculations complete the proof of \cref{thm:moduli} in
case (1).  Consequently, we may alternatively formulate case (1)
of the theorem as saying that $(I_2 : I_4 : I_6 : I_{10})$ is of one
of the following forms:

\begin{enumerate}
\item[(1a)] 
\begin{multline} \label{eq:mnICs}
 (-20 (-2 m^{2} + 10 n^{2} + 3) :  
25(-m^{2} + 5 n^{2} + 9)^{2} : \\
-5(-75 m^{6} + 1125 m^{4} n^{2} - 5625 m^{2} n^{4} + 9375 n^{6} - 72 m^{5} + 720 m^{3} n^{2} - 1800 m n^{4} \\
+ 1165 m^{4} - 11650 m^{2} n^{2} + 29125 n^{4} + 360 m^{3} - 1800 m n^{2} - 5985 m^{2} + 29925 n^{2} + 6399) : \\
8 (m^{5} - 10 m^{3} n^{2} + 25 m n^{4} + 5 m^{4} - 50 m^{2} n^{2} + 125 n^{4} - 5 m^{3} + 25 m n^{2} - 45 m^{2} + 225 n^{2} + 108)^{2})
\end{multline}
where $(m,n) \in k^2$ such that  $m^2-5n^2-5$ is a norm from $k(\sqrt 5)$,
$n(m^2-5n^2)(m^2-5n^2-5) \ne 0$ and
$8 m^{5} - 80 m^{3} n^{2} + 200 m n^{4} - 85 m^{4} + 850 m^{2} n^{2} - 2125 n^{4} - 40 m^{3} + 200 m n^{2} + 1890 m^{2} - 9450 n^{2} - 9261 \ne 0$;

\item[(1b)] $(-12: 81 : 36000h - 567 : 400000h^2)$ if $\sqrt 5 \in k$
and $h \in k \setminus \{ 0, \frac{81}{800}, \frac{27}{3125} \}$; or

\item[(1c)] $(14 : 4 : 4500h + 16 : 12500h^2)$
if $\sqrt 5 \in k$ and $h \in k \setminus \{ 0, \frac{1}{50}, \frac{2}{3125} \}$.
\end{enumerate}

In particular, when $k=\Q$, we can deduce the following precise 
interpretation of \cref{thm:main} from \cref{thm:moduli}: The
set of all genus 2 curves $C$ with RM-5 over $\Q$ up to 
$\C$-isomorphism such that $\Aut_\C^{\red}(C) = \{ 1 \}$
excluding the 1-parameter family in case (2) correspond to
points $(m,n)$ with Igusa--Clebsch invariants as in (1a).  
Moreover, each tuple of Igusa--Clebsch invariants
as in (1a) comes from such a curve, except possibly when
these Igusa--Clebsch invariants lead to a non-commutative 
endomorphism algebra, in which case we only know that
such $(m,n)$ corresponds to a curve 
defined over $\Q$ with potential RM-5.
Further, distinct points $(m,n)$ as in (1a) correspond
to distinct $\C$-isomorphism classes of genus 2 curves
by \cref{prop:necess-cond}.

\subsection{Singularities of the Mestre conic} \label{sec:mest-sing}
Here we describe the $k$-rational Igusa--Clebsch invariants
of types (2) and (3) in \cref{prop:necess-cond} for which the
the Mestre conic is singular.  By \cite{cardona-quer}, these
invariants always yield a genus 2 curve defined over $k$.
This will complete the proof of \cref{thm:moduli}.

First, as in \cref{subsection:conic-reduction-generic}, 
let $L$ be the Mestre conic associated to a $(z,g,h) \in Y$.
Then the Mestre conic is singular if and only if $h=0$, 
$h=\frac {9g^2}8$ or $z=0$.  

We remark that the curve $h=0$ on $Y$ is given by 
$z^2 = -54 (6g + 1)^2  g^3$,
and the $k$-rational points are parametrized by $(g,0)$ where
$-6g$ is a square in $k$.  However, $h=0$ means that $I_{10} = 0$, so
these points do not correspond to genus 2 curves.

The curve $h=\frac {9g^2}8$ on $Y$ is given by 
$z^2 = -\frac{27}{16}  (128g + 9)  g^2  (3g - 4)^2,$
and the $k$-rational points are parametrized by 
$(g, \frac{9g^2}8)$ where $-3(128g+9)$ is a square in $k$.
This completes case (3) of the theorem.

Now suppose $z=0$, which means that either
$g \in \{ - \frac 3{10}, - \frac 2{15} \}$ or $n=0$.  
If $g = -\frac 3{10}$ or $g = -\frac 2{15}$,
then $h = \frac{27}{3125}$ or $h = \frac{2}{3125}$, respectively,
and these are clearly $k$-rational points on $Y$.
The corresponding Igusa--Clebsch invariants are
$(20 : 225 : 1185 : -384)$ and $(70 : 100 : 2360 : 16)$,
respectively.  
As notes in \cite{EK},
the $k$-rational points on $Y$ with $n=0$ are given by
\[ (z, g, h) = \left(0, \frac{m^2-9}{30}, \frac{(m-2)^2(m+3)^3}{12500}\right),
\quad m \in k. \]
Viewing this as a map from points $(m,0)$ to $(0,g,h)$,
note that $m = 0$ and $m = \pm \sqrt 5$
respectively map to $(g,h) = (-\tfrac{3}{10}, \tfrac{27}{3125})$ and $(-\tfrac{2}{15}, \tfrac{2}{3125})$.
This gives case (4) of the theorem.

Now we consider the ``points at infinity'' discussed in \cref{sec:infty}.
The Mestre conic associated to Igusa--Clebsch invariants $(8 : 1 : 3: s)$
for $s \in k^\times$ is singular if and only if $s = \frac 8{3125}$.
In terms of Wilson's moduli, this point corresponds to
$(z_6, s_2, \sigma_5) = (0, -\frac 52, \frac 12)$.  Here Wilson's 
discriminant $\Delta'$ is 0.  Using Magma, we can construct
a rational genus 2 curve with invariants $(8 : 1 : 3 : \frac 8{3125})$,
namely
\begin{equation} \label{eqn:813s-curve}
 y^2 = f(x) = (2 x^{3} - 2 x^{2} - x - 1)(x^{3} - x^{2} + 2 x + 2). 
\end{equation}
This yields case (5) of the theorem.

\begin{rem} Calculations in Magma indicate that the curve in
\eqref{eqn:813s-curve} has conductor $800^2$, and corresponds
to the weight 2 modular form $f(z) = q - \sqrt 5 q^3 - 2 \sqrt 5 q^7 +
2 q^9 - \sqrt 5 q^{11} + \cdots$ with Fourier coefficient ring $\Z[\sqrt 5]$
and LMFDB label \verb+800.2.a.l+.
\end{rem}

%
%

\section{Generic models}
\label{sec:models}

In this section we give explicit rational Weierstrass models for 
$(m,n)$ in the birational model $\PP^2_{m,n}$ for $Y_-(5)$.

\begin{prop} \label{prop:model}
  Let $k \subseteq \C$ be a field which does not contain $\sqrt{5}$. For any $m,n \in k$ such that $-(m^2 - 5n^2 - 5)$ is the norm of some nonzero element $\eta \in k(\sqrt{5})/k$, let $\mu \coloneqq m + n\sqrt{5}$ and define $C/k$ be the curve with Weierstrass model
  \begin{align*}
    y^2 = \text{\emph{Tr}}\!\left(\mu^2\eta^3\!\left(\frac{1-x\sqrt{5}}{1+x\sqrt{5}}\right)^3 - 2N(\mu)\mu\eta^2\!\left(\frac{1-x\sqrt{5}}{1+x\sqrt{5}}\right)^2 - 5N(\mu)(N(\mu)-5)\right)(1 - 5x^2)^3,
  \end{align*}
  where $N$ and $\text{\emph{Tr}}$ denote the norm and trace from $k(\sqrt{5})$ to $k$ respectively. When $C$ is a genus $2$ curve, the Igusa--Clebsch invariants of $C$ are as in \eqref{eq:mnICs}, 
  i.e. $C$ corresponds to the point $(m,n)$ in
  the $\PP^2_{m,n}$ birational model for $Y_-(5)$.
\end{prop}

Note that the right hand side of the Weierstrass equation given in \cref{prop:model} is indeed a sextic in $x$; the factor of $(1 - 5x^2)^3$ clears denominators.

\begin{rem} Since $m^2 - 5n^2 - 5 + u^2 - 5v^2 = 0$ is a quadric in $\PP^4$
it is birational to $\PP^3$, so one may generically express the family of
curves in \cref{prop:model} in terms of a 3-parameter $(a,b,c)$.  For instance,
one may generically write 
\[ v = (4a+2c)/(5a^2-b^2+5c^2-1), \quad
m = 5av - 2, \quad n=-bv, \quad u = 5cv-1 \]
to get a 3-parameter family of genus 2 curves with RM-5.  
However, the
resulting models are rather complicated and we omit them.
\end{rem}

\begin{prop} \label{prop:modelQsqrt5}
  Let $k \subseteq \C$ be a field containing $\sqrt{5}$. For any $m,n \in k$, define $C/k$ be the curve with Weierstrass model
  \begin{align*}
    y^2 =\quad &(m - n\sqrt{5})^2x^6 - 2(m^2 - 5n^2)(m - n\sqrt{5})x^5 - 10(m^2 - 5n^2)(m^2 - 5n^2 - 5)x^3\\
    -&2(m^2 - 5n^2)(m^2 - 5n^2 - 5)^2(m + n\sqrt{5})x - (m^2 - 5n^2 - 5)^3(m + n\sqrt{5})^2.
  \end{align*}
When $C$ is a genus $2$ curve, the Igusa--Clebsch invariants of $C$ are as in \eqref{eq:mnICs}, 
  i.e. $C$ corresponds to the point $(m,n)$ in
  the $\PP^2_{m,n}$ birational model for $Y_-(5)$.
\end{prop}

Note that both the $x^4$- and $x^2$-coefficients are zero in this model.

\begin{proof}[Proof of \cref{prop:model,prop:modelQsqrt5}]
In \cref{subsection:conic-reduction-generic}, we found a linear transformation $T$ defined over $\Q(m,n)$ and a scaling factor $c \in \Q(m,n)$ such that
$$cL_0(T(x_1,x_2,x_3)) = x_1^2 - 5x_2^2 + (m^2 - 5n^2 - 5)x_3^2,$$
whenever $\text{disc}\,L_0 \neq 0$, where $L_0$ is the Mestre conic associated to the Igusa--Clebsch invariants in \eqref{eq:mnICs}.
 Applying the same transformation $T$ to the Mestre cubic $M_0$ and rescaling by some $c' \in \Q(m,n)$ yields
\begin{align*}
  c'M_0(T(x_1,x_2,x_3)) = \quad & (m+5n^2)x_1^3 + 30mnx_1^2x_2 + 15(m^2+5n^2)x_1x_2^2 + 50mnx_2^3\\
  &-(2m-3)(m^2-5n^2)x_1^2x_3 - 20n(m^2-5n^2)x_1x_2x_3  \\
 &-5(2m+3)(m^2-5n^2)x_2^2x_3 - 2(m^2-5n^2-5)(m^2-5n^2)x_3^3.
\end{align*}
Define $L$ and $M$ to be these reduced forms of $L_0$ and $M_0$ respectively.

We first consider the case where $\sqrt{5} \not\in k$. By inspection, $L(k)$ has no points with $x_3 = 0$. Suppose that $(u_0\,:\,v_0\,:\,1) \in L(k)$. Parametrizing $L(k)$ in the usual way using this point gives
\begin{equation*} \label{eq:mestre-conic-param}
  \begin{aligned}
    &\left\{ (x_1\,:\,x_2\,:\,1) \in L(k) \right\} = \left\{ \left((1+5x^2)u_0 - 10xv_0 \,:\, (1+5x^2)v_0 - 2xu_0 \,:\, 1 - 5x^2\right)\,:\, x\in\PP(k)\right\}.
  \end{aligned}
\end{equation*}
It will be convenient for us to write this parametrization in terms of elements of $k(\sqrt{5})$. Define $\eta = u_0 + v_0\sqrt{5} \in k(\sqrt{5})/k$. The parametrization above can then be expressed as
$$\left\{ (x_1\,:\,x_2\,:\,1) \in L(k) \right\} = \left\{ \left(u \,:\, v \,:\, 1\right)\,:\, u + \sqrt{5}v = \eta\frac{1 - x\sqrt{5}}{1 + x\sqrt{5}},\,x\in\PP(k)\right\}.$$
Let $\mu = m + n\sqrt{5} \in k(\sqrt{5})/k$. Then one verifies that, when $x_3 = 1$, the reduced Mestre cubic $M$ can be written as
$$\tfrac{1}{2}\text{Tr}(\mu^2(x_1 + x_2\sqrt{5})^3) + 3N(\mu)N(x_1 + x_2\sqrt{5}) - N(\mu)\text{Tr}(\mu(x_1 + x_2\sqrt{5})^2) - 2N(\mu)(N(\mu)-5),$$
where $N$ and $\text{Tr}$ denote the norm and trace from $k(\sqrt{5})$ to $k$. We can now substitute the parametrization of the $k$-rational points of $L$ into $M$ to obtain a $k$-rational Weierstrass model for the associated genus $2$ curve $C$, as described in  \cref{thm:mestre-cubic}. This gives the $k$-rational model from  \cref{prop:model}.

Now suppose that $\sqrt{5} \in k$. It is possible to mimic the calculations from the case where $\sqrt{5}\not\in k$ by taking
$$(u_0\,:\,v_0\,:\,1) \coloneqq \left(\left(m^2 - 5n^2 - 5 - \tfrac{1}{20}\right)\!\sqrt{5} \,:\, m^2 - 5n^2 - 5 + \tfrac{1}{20} \,:\, 1\right)$$
and working in the ring $k[t]/(t^2 - 5)$. However, we get a tidier Weierstrass model by instead using the point $(\sqrt{5}\,:\,1\,:\,0)$ to parameterize $L(k)$. A straightforward calculation yields the $k$-rational model given in  \cref{prop:modelQsqrt5}.
\end{proof}

\begin{rem} Note that if we take $k=\Q(\sqrt 5)$ in 
\cref{prop:modelQsqrt5}, and $m \in \Q$, $n \in \sqrt 5 \Q \setminus \{ 0 \}$,
we get a 2-parameter family of genus 2 curves defined over $\Q$
which have potential RM-5, but not RM-5 defined over $\Q$.  
To see the RM-5 is not actually defined over $\Q$, one can check that  for
$m \in \Q$, $n \in \sqrt 5 \Q  \setminus \{ 0 \}$, Wilson's discriminant $\Delta'$
is in the square class of $n^2$, which is a non-square.
Hence the Igusa--Clebsch invariants are rational, but the
moduli points on $Y_-(5)$ are not rational, and so the RM-5
cannot be defined rationally.
(The irrationality of
these moduli points on $Y_-(5)$ is suggested by the fact
that $(m,n) \not \in \Q^2$ but is not \textit{a priori} implied by this
as $(m,n)$ are only coordinates for a birational model of $Y_-(5)$,
and we have not determined an explicit birational map from
$\PP^2_{m,n}$ to $Y_-(5)$).
\end{rem}

%
%

\section{Comparisons with known families}
\subsection{Mestre's family}
Let $f$ be the polynomial
$$f(a,b,x) = x^5 + (a-3)x^4 + (-a + b + 3)x^3 + (a^2 - a - 2b - 1)x^2 + bx + a,$$
and let $X(a,b)$ be the genus $2$ curve
$$X(a,b): y^2 = xf(a,b,x).$$
In \cite{mestre}, Mestre proves that $X(a,b)$ has RM-5 for every $a,b$ in $\C$ such that $xf(a,b,x)$ has six distinct zeroes, and that the RM is defined over 
$k = \Q(a,b)$.  Using Humbert's criterion for RM-5, Wilson \cite{wilson}
showed that this family of curves over $k$ gives all genus 2 curves with RM-5
over $k$ which have a Weierstrass point in $k$, up to $k$-isomorphism.  
In particular, for any genus $2$ curve $C$ with RM-5, there exist $a,b \in \C$ such that $C$ is $\C$-isomorphic to $X(a,b)$.  See also \cite{sakai}
for an alternative proof of this last result.

Define $g_{a,b}$ and $h_{a,b}$ as
\begin{align*}
  &g_{a,b} = \frac{2(3a^3 - 8a^2 - 5ab - b^2 - 3a)}{3(a^2 - 5a - 2b + 1)^2}\\
  &\text{and}\\
  &h_{a,b} = \frac{-a^2(4a^5 - 4a^4 - 24a^3b - a^2b^2 - 40a^3 + 34a^2b + 30ab^2 + 4b^3 + 91a^2 + 14ab - b^2 - 4a)}{2(a^2 - 5a - 2b + 1)^5}.
\end{align*}
Then, by comparing Igusa--Clebsch invariants, one can verify that $X(a,b)$ is $\C$-isomorphic to a genus $2$ curve associated to $(g,h) = (g_{a,b}, h_{a,b})$ in the Elkies--Kumar model \eqref{eq:EK}, assuming
$ a^2 - 5a - 2b + 1 \ne 0$ and $h_{a,b} \ne 0$.

Since choices of $a, b \in k$ can only yield genus 2 curves with RM-5 over 
$k$ which have a $k$-rational Weierstrass point, one cannot easily describe
all genus 2 curves with RM-5 over $k$ using Mestre's family.
For example, there are no rational values of $a$ and $b$ for which 
$X(a,b)$ is $\C$-isomorphic to the genus $2$ curve associated to 
$(g,h) = (-\tfrac{4}{15}, \tfrac{16}{3125})$.

\subsection{Brumer's family}
Brumer constructed a family of curves 
$C_{b,c,d}$ defined by
\begin{align*}
  C_{b,c,d}: y^2 + (x^3 + x + 1 + c(x^2 + x))y =\,\,&b + (1 + 3b)x + (1 - bd + 3b)x^2\\
  +&(b - 2bd - d)x^3 - bdx^4,
\end{align*}
and showed that if $C_{b,c,d}$ is nonsingular, then it is a genus
2 curve with RM-5 over $\Q(b,c,d)$.  Moreover, every genus $2$ 
curve with RM-5 is $\C$-isomorphic to $C_{b,c,d}$ for some $b,c,d \in \C$.
Brumer did not publish the details of his proof 
(see \cite{brumer} for an announcement), but the above statements
were reproved by different methods in \cite{hashimoto} and 
\cite{hashimoto-sakai}.

Define $g_{b,c,d}$ and $h_{b,c,d}$ as
\begin{align*}
  &g_{b,c,d} = \frac{-c^4 + 8bc^2d - 16b^2d^2 + 6c^3 - 24bcd + 24bc + c^2 - 68bd - 24cd - 108b - 30c - 36d - 61}{6(c^2 - 4bd - 2b - 3c - 2d - 5)^2}\\
\end{align*}
and
\begin{align*}
  h_{b,c,d} = \,&(c^2 - 4bd - 2b - 3c - 2d - 5)^{-5}  \left(bc^6d - 12b^2c^4d^2 + 48b^3c^2d^3 - 64b^4d^4 - b^2c^4d - 9bc^5d \right.\\
  &+ 8b^3c^2d^2 + 72b^2c^3d^2 - bc^4d^2 - 16b^4d^3 - 144b^3cd^3 + 8b^2c^2d^3 - 16b^3d^4 + bc^5 - 40b^2c^3d\\
  &+ 12bc^4d - c^5d + 144b^3cd^2 - 152b^2c^2d^2 + 52bc^3d^2 + 416b^3d^3 - 192b^2cd^3 - b^2c^3 - 9bc^4\\
  &+ 36b^3cd + 334b^2c^2d + 63bc^3d + 6c^4d + 24b^3d^2 + 132b^2cd^2 - 80bc^2d^2 + c^3d^2 + 528b^2d^3\\
  &- 36bcd^3 - 27b^2c^2 + 13bc^3 - c^4 + 108b^3d - 720b^2cd + 74bc^2d + 5c^3d - 456b^2d^2 - 96bcd^2\\
  &- 36c^2d^2 + 216bd^3 + 27b^3 + 252b^2c + 56bc^2 + 6c^3 - 66b^2d - 627bcd - 43c^2d - 381bd^2\\
  &- \left.63cd^2 + 27d^3 - 567b^2 + 27bc + 4c^2 - 121bd - 147cd - 81d^2 - 484b - 39c - 34d - 103\right).
\end{align*}
Then, by comparing Igusa--Clebsch invariants, one can verify that $C_{b,c,d}$ is $\C$-isomorphic to the genus $2$ curve associated to $(g,h) = (g_{b,c,d}, h_{b,c,d})$ in the Elkies--Kumar model \eqref{eq:EK} when $c^2 - 4bd - 2b - 3c - 2d - 5 \ne 0$ and $h_{b,c,d} \ne 0$.

One can ask if Brumer's family provides a way to describe all
 genus 2 curves $C$ with RM-5 defined over $k$.  However, it
 is not clear whether these will all come from a $k$-rational
 choice of parameters $b, c, d$.  E.g., if $(z,g,h)$ is a generic
 rational point on $Y$ such that $30g+4$ is a norm from $\Q(\sqrt 5)$,
 it is not clear if we can write $(g,h) = (g_{b,c,d}, h_{b,c,d})$ for
 some $b, c, d \in \Q$.
 
 While Brumer's models are simpler than what we give
 in \cref{sec:models}, over $\Q$ they might not comprise all 
 rational curves $C$ with RM-5, even generically.  Moreover
 there is no simple description of which choices of $b, c, d$
 will give $\C$-isomorphic curves.

%
%

\section{Beyond RM-5} \label{sec:otherD}
The Hilbert modular surface $Y_-(D)$ is rational if and only if $D$ is one of $5, 8, 12, 13,$ or $17$. One might wonder if there are analogues of \cref{thm:main} for each of these discriminants. Numerical experimentation suggests that the answer is yes.

Define
\begin{align*}
  &p_5(m,n) = -m^2 + 5n^2 + 5,\\
  &p_8(m,n) = m+1,\\
  &p_{12}(m,n) = -27m^2 + n^2 + 27,\\
  &p_{13}(m,n) = 1803m^2 - 72mn + n^2 + 3168m - 1440n - 768,\,\text{and}\\
  &p_{17}(m,n) = 1.
\end{align*}
In \cite{EK}, Elkies and Kumar give rational models of $Y_-(D)$ for all fundamental discriminants between $1$ and $100$. The polynomials $p_D(m,n)$ above are all factors of the discriminant of the Mestre conic one obtains when using Igusa--Clebsch invariants from \cite{EK} in the construction given in \cref{section:mestre-conic-background}. We chose several thousand values of $(m,n) \in \Q^2$ at random, and for each of these the associated Mestre conic was equivalent to $x_1^2 - Dx_2^2 - p_D(m,n)x_3^2 = 0$ over $\Q$ whenever it was nonsingular. In particular, the Mestre obstruction appears to vanish generically for $D = 17$, which is quite surprising.

We have attempted using the methods from
in \cref{sec:reduction} to reduce the Mestre conics for these other
values of $D$, but thus far have only been partially successful
in removing the other polynomial factors from the discriminant.

%
%


\begin{bibdiv}
\begin{biblist}

\bib{brumer}{article}{
   author={Brumer, Armand},
   title={The rank of $J_0(N)$},
   note={Columbia University Number Theory Seminar (New York, 1992)},
   journal={Ast\'{e}risque},
   number={228},
   date={1995},
   pages={3, 41--68},
   issn={0303-1179},
}

\bib{cardona-quer}{article}{
   author={Cardona, Gabriel},
   author={Quer, Jordi},
   title={Field of moduli and field of definition for curves of genus 2},
   conference={
      title={Computational aspects of algebraic curves},
   },
   book={
      series={Lecture Notes Ser. Comput.},
      volume={13},
      publisher={World Sci. Publ., Hackensack, NJ},
   },
   date={2005},
   pages={71--83},
}

\bib{CMSV}{article}{
   author={Costa, Edgar},
   author={Mascot, Nicolas},
   author={Sijsling, Jeroen},
   author={Voight, John},
   title={Rigorous computation of the endomorphism ring of a Jacobian},
   journal={Math. Comp.},
   volume={88},
   date={2019},
   number={317},
   pages={1303--1339},
   issn={0025-5718},
}

\bib{elkies}{article}{
   author={Elkies, Noam D.},
   title={Shimura curve computations via $K3$ surfaces of N\'{e}ron-Severi rank
   at least 19},
   conference={
      title={Algorithmic number theory},
   },
   book={
      series={Lecture Notes in Comput. Sci.},
      volume={5011},
      publisher={Springer, Berlin},
   },
   date={2008},
   pages={196--211},
}

\bib{EK}{article}{
   author={Elkies, Noam},
   author={Kumar, Abhinav},
   title={K3 surfaces and equations for Hilbert modular surfaces},
   journal={Algebra Number Theory},
   volume={8},
   date={2014},
   number={10},
   pages={2297--2411},
   issn={1937-0652},
}

\bib{hashimoto}{article}{
   author={Hashimoto, Ki-ichiro},
   title={On Brumer's family of RM-curves of genus two},
   journal={Tohoku Math. J. (2)},
   volume={52},
   date={2000},
   number={4},
   pages={475--488},
   issn={0040-8735},
}

\bib{hashimoto-sakai}{article}{
   author={Hashimoto, Kiichiro},
   author={Sakai, Yukiko},
   title={General form of Humbert's modular equation for curves with real
   multiplication of $\Delta=5$},
   journal={Proc. Japan Acad. Ser. A Math. Sci.},
   volume={85},
   date={2009},
   number={10},
   pages={171--176},
   issn={0386-2194},
}

\bib{KW}{article}{
   author={Khare, Chandrashekhar},
   author={Wintenberger, Jean-Pierre},
   title={Serre's modularity conjecture. I},
   journal={Invent. Math.},
   volume={178},
   date={2009},
   number={3},
   pages={485--504},
   issn={0020-9910},
}

\bib{KM}{article}{
   author={Kumar, Abhinav},
   author={Mukamel, Ronen E.},
   title={Real multiplication through explicit correspondences},
   journal={LMS J. Comput. Math.},
   volume={19},
   date={2016},
   number={suppl. A},
   pages={29--42},
}

\bib{magma}{article}{
   author={Bosma, Wieb},
   author={Cannon, John},
   author={Playoust, Catherine},
   title={The Magma algebra system. I. The user language},
   note={Computational algebra and number theory (London, 1993)},
   journal={J. Symbolic Comput.},
   volume={24},
   date={1997},
   number={3-4},
   pages={235--265},
   issn={0747-7171},
   label={Magma}
}

\bib{mestre:family}{article}{
   author={Mestre, J.-F.},
   title={Familles de courbes hyperelliptiques \`a multiplications r\'{e}elles},
   language={French},
   conference={
      title={Arithmetic algebraic geometry},
      address={Texel},
      date={1989},
   },
   book={
      series={Progr. Math.},
      volume={89},
      publisher={Birkh\"{a}user Boston, Boston, MA},
   },
   date={1991},
   pages={193--208},
}

\bib{mestre}{article}{
   author={Mestre, Jean-Fran\c{c}ois},
   title={Construction de courbes de genre $2$ \`a partir de leurs modules},
   language={French},
   conference={
      title={Effective methods in algebraic geometry},
      address={Castiglioncello},
      date={1990},
   },
   book={
      series={Progr. Math.},
      volume={94},
      publisher={Birkh\"{a}user Boston, Boston, MA},
   },
   date={1991},
   pages={313--334},
}

\bib{ribet}{article}{
   author={Ribet, Kenneth A.},
   title={Abelian varieties over $\bf Q$ and modular forms},
   conference={
      title={Modular curves and abelian varieties},
   },
   book={
      series={Progr. Math.},
      volume={224},
      publisher={Birkh\"{a}user, Basel},
   },
   date={2004},
   pages={241--261},
}

\bib{sage}{manual}{
      author={Developers, The~Sage},
       title={{S}agemath, the {S}age {M}athematics {S}oftware {S}ystem
  ({V}ersion 9.4)},
        date={2021},
        label={Sage},
        note={{\tt https://www.sagemath.org}},
}

\bib{sakai}{article}{
   author={Sakai, Yukiko},
   title={Poncelet's theorem and curves of genus two with real
   multiplication of $\Delta=5$},
   journal={J. Ramanujan Math. Soc.},
   volume={24},
   date={2009},
   number={2},
   pages={143--170},
   issn={0970-1249},
}

\bib{vdG}{book}{
   author={van der Geer, Gerard},
   title={Hilbert modular surfaces},
   series={Ergebnisse der Mathematik und ihrer Grenzgebiete (3) [Results in
   Mathematics and Related Areas (3)]},
   volume={16},
   publisher={Springer-Verlag, Berlin},
   date={1988},
   pages={x+291},
   isbn={3-540-17601-2},
}

\bib{wilson}{article}{
   author={Wilson, John},
   title={Explicit moduli for curves of genus 2 with real multiplication by
   ${\bf Q}(\sqrt5)$},
   journal={Acta Arith.},
   volume={93},
   date={2000},
   number={2},
   pages={121--138},
   issn={0065-1036},
}


\end{biblist}
\end{bibdiv}
\end{document}